\documentclass[12pt,a4paper]{article}

\usepackage{graphicx}
\usepackage{subfig}
\usepackage{amsmath,amsthm,amssymb}
\usepackage{esint}
\usepackage{mathrsfs}
\usepackage{appendix}

\usepackage{hyperref}

\usepackage[left=2.2cm,right=2.2cm,top=3cm,bottom=3.5cm]{geometry}

\usepackage{color}
\usepackage{array}
\usepackage{url}

\usepackage{float}

\DeclareMathOperator{\supp}{supp}

\newtheorem{theorem}{Theorem}[section]
\newtheorem*{theorem*}{Theorem}
\newtheorem{lemma}[theorem]{Lemma}
\newtheorem{corollary}[theorem]{Corollary}
\newtheorem{proposition}[theorem]{Proposition}

\theoremstyle{definition}
\newtheorem{definition}[theorem]{Definition}

\theoremstyle{remark}
\newtheorem{remark}[theorem]{Remark}
\usepackage{authblk}

\usepackage{mathtools}
\usepackage{amssymb}
\usepackage{appendix}

\numberwithin{equation}{section}

\usepackage[numbers,sort&compress]{natbib}

\DeclareMathOperator{\Div}{div}

\allowdisplaybreaks

\begin{document}
	
	\title{On the conservation of helicity by weak solutions of the 3D Euler and inviscid MHD equations}

	\author[]{Daniel W. Boutros\footnote{Department of Applied Mathematics and Theoretical Physics, University of Cambridge, Cambridge CB3 0WA UK. Email: \textsf{dwb42@cam.ac.uk}} \space and Edriss S. Titi\footnote{Department of Mathematics, Texas A\&M University, College Station, TX 77843-3368, USA; Department of Applied Mathematics and Theoretical Physics, University of Cambridge, Cambridge CB3 0WA UK; also Department of Computer Science and Applied Mathematics, Weizmann Institute of Science, Rehovot 76100, Israel. Emails: \textsf{titi@math.tamu.edu} \; \textsf{Edriss.Titi@maths.cam.ac.uk}}}

	\date{July 31, 2025} 
	
	\maketitle
	
\begin{abstract}
Classical solutions of the three-dimensional Euler equations of an ideal incompressible fluid conserve the helicity. We introduce a new weak formulation of the vorticity formulation of the Euler equations in which (by implementing the Bony paradifferential calculus) the advection terms are interpreted as paraproducts for weak solutions with low regularity. Using this approach we establish an equation of local helicity balance, which gives a rigorous foundation to the concept of local helicity density and flux at low regularity. We provide a sufficient criterion for helicity conservation which is weaker than many of the existing sufficient criteria for helicity conservation in the literature.

Subsequently, we prove a sufficient condition for the helicity to be conserved in the zero viscosity limit of the Navier-Stokes equations. Moreover, we establish a relation between the defect measure (which is part of the local helicity balance) and a third-order structure function for solutions of the Euler equations. As a byproduct of the approach introduced in this paper, we also obtain a new sufficient condition for the conservation of magnetic helicity in the inviscid MHD equations, as well as for the kinematic dynamo model. 

Finally, it is known that classical solutions of the ideal (inviscid) MHD equations which have divergence-free initial data will remain divergence-free, but this need not hold for weak solutions. We show that weak solutions of the ideal MHD equations arising as weak-$*$ limits of Leray-Hopf weak solutions of the viscous and resistive MHD equations remain divergence-free in time. 
\end{abstract}


\vspace{0.1cm} \noindent \textbf{Mathematics Subject Classification:} 35Q35 (primary), 35Q31, 76F02, 76B99, 76W05, 35D30 (secondary)

\section{Introduction}
Classical solutions of the equations of motion of inviscid incompressible flows possess several conserved quantities, such as the energy and the helicity. In this work, we focus on establishing sufficient conditions for the regularity of weak solutions to conserve the helicity. We recall that the Euler equations of an ideal incompressible fluid subject to periodic boundary conditions on the three-dimensional torus $\mathbb{T}^3$ are given by
\begin{equation} \label{eulerequation}
\partial_t u + \nabla \cdot (u \otimes u) + \nabla p = 0, \quad \nabla \cdot u = 0,
\end{equation}
where $u : \mathbb{T}^3 \times (0,T) \rightarrow \mathbb{R}^3$ is the velocity field and $p : \mathbb{T}^3 \times (0,T) \rightarrow \mathbb{R}$ is the pressure. By taking the curl of the momentum equation, the Euler equations can be stated in the vorticity formulation, which is
\begin{equation} \label{vorticityequation}
\partial_t \omega + u \cdot \nabla \omega - \omega \cdot \nabla u = 0, 
\end{equation} 
where $\omega = \nabla \times u$ is the vorticity, which is also an incompressible vector field. We recall that the velocity field $u$ can be recovered from the vorticity $\omega$ by means of the Biot-Savart law
\begin{equation}
u = \nabla \times (-\Delta)^{-1} \omega.
\end{equation}
We will assume throughout that $\omega$ is mean-free, such that $\int_{\mathbb{T}^3} \omega dx = 0$. Moreover, we define $(-\Delta)^{-1} \omega$ to be a mean-free function.

For the Euler equations the helicity is given by
\begin{equation}
\mathcal{H} (t) = \int_{\mathbb{T}^3} u (x,t) \cdot \omega (x,t) dx,
\end{equation}
We will write $h(x,t) \coloneqq u(x,t) \cdot \omega(x,t)$ for the helicity density. 
One can check, see for example \cite[~Section 2.1]{derosa}, that classical solutions of the Euler equations conserve the helicity. The concept of helicity as a conserved quantity for inviscid flows was introduced independently in the papers \cite{moffatt} and \cite{moreau}, which pointed out its fundamental topological importance (see also \cite{moffattreview}). 

The helicity describes the knottedness of vortex tubes and vortex lines and is therefore of importance in topological fluid mechanics. In particular, (local) helicity conservation implies that the topological configuration (i.e. the linkage) of vortex lines cannot change in time \cite{davidson}. More generally speaking, it means that the topological properties of the vorticity field are invariant under the dynamics \cite{moffattreview}. 

Helicity is also an important quantity in the study of three-dimensional turbulence, as it is the only other known conserved quantity for inviscid flow in addition to the kinetic energy. One of the effects of helicity dynamics is that it can impede the forward energy cascade, as noted in \cite{moffatt2014} (see also \cite{yan,biferale2012}). 

In a recent work \cite{boutroshydrostatic}, two new types of weak solutions were introduced for the hydrostatic Euler equations (also known as the inviscid primitive equations of oceanic and atmospheric dynamics), in which the nonlinearity was interpreted as a paraproduct in a negative Besov space by using the Bony paradifferential calculus. The existence of a version of these weak solutions to the hydrostatic Euler equations was established in \cite{boutrosnonuniqueness}.

In this work we would like to extend this idea to the problem of helicity conservation. In particular, we are interested in the case where the vorticity lies in negative Sobolev (or Besov) spaces and the velocity field has sufficient regularity such that the nonlinear terms in the vorticity equation \eqref{vorticityequation} can be interpreted as paraproducts. 

In recent years the conservation of kinetic energy by weak solutions of the incompressible Euler equations has received a lot of attention in the literature. It was conjectured by Lars Onsager in \cite{onsager} that there exists a critical regularity threshold for energy conservation. The positive part of this conjecture was investigated and established in \cite{eyink,constantin,duchon,cheskidovconstantin}, while for the negative part there are the works \cite{lellisinclusion,lellisadmissibility,isettproof,buckmasteradmissible} (and see references therein). 

Since the helicity is also a conserved quantity of the Euler equations, it is natural to investigate the sufficient regularity assumptions for weak solutions of the Euler equations to conserve the helicity. In \cite{chae} it was shown that helicity conservation holds if $\omega$ belongs to the space $L^3 ((0,T); B^{1/3+}_{9/5,\infty} (\mathbb{T}^3))$, these results were then improved in \cite{chaeconserved}. In \cite{cheskidovconstantin} it was proven that the helicity is conserved if $u \in L^3 ((0,T); B^{2/3}_{3,c(\mathbb{N})} (\mathbb{T}^3))$ (where $c (\mathbb{N})$ indicates an assumption on the decay rate of the Littlewood-Paley blocks, see \cite{cheskidovconstantin} for more details). This condition was slightly improved in \cite{wanghelicity}.

It is also possible to derive sufficient criteria by making independent regularity assumptions on both the velocity and vorticity fields, which was done in \cite{derosa,wangcompressible}. It was shown in \cite{derosa,liuholder} that under sufficient regularity assumptions on the velocity field, the helicity is H\"older continuous in time (if it is not conserved). The asymptotic behaviour of the helicity for deterministic and statistical solutions of the three-dimensional Navier-Stokes equations was studied in \cite{foias2007,foias2009}. In \cite{berselli2009} a regularity criterion for the three-dimensional Navier-Stokes equations based on the vanishing of the helicity density was established.

In this paper, we will also consider the incompressible ideal magnetohydrodynamic (MHD) equations, which are given by
\begin{align}
&\partial_t u + u \cdot \nabla u - B \cdot \nabla B + \nabla p = 0, \label{MHD1} \\ 
&\partial_t B + u \cdot \nabla B - B \cdot \nabla u = 0, \label{MHD2} \\
&\nabla \cdot u = 0, \label{MHD3}
\end{align}
where $u: \mathbb{T}^3 \times (0,T) \rightarrow \mathbb{R}^3 $ is again the velocity field and $B : \mathbb{T}^3 \times (0,T) \rightarrow \mathbb{R}^3$ is the magnetic field. We recall the definition of the vector potential $A : \mathbb{T}^3 \times (0,T) \rightarrow \mathbb{R}^3$, which is defined by the relation $\nabla \times A = B$ (up to a gradient term $\nabla \varphi$). We then choose $A$ to be divergence-free (which implies $\varphi \equiv 0$ and hence determines $A$ uniquely up to a constant), and we have the relation (again up to a constant)
\begin{equation} \label{potentialequation}
A = \nabla \times (-\Delta)^{-1} B,
\end{equation}
where we assume $B$ to be mean-free, and we fix $(-\Delta)^{-1} B$ to be a mean-free function. We also recall that the inviscid and irresistive MHD equations in potential form are given by
\begin{align}
&\partial_t A + B \times u + \nabla \tilde{\phi} = 0, \label{mhdpotential1} \\
&\partial_t u + u \cdot \nabla u - B \cdot \nabla B + \nabla p = 0, \label{mhdpotential2} \\
&\nabla \cdot u = \nabla \cdot A = 0, \label{mhdpotential3}
\end{align}
where $\tilde{\phi}$ is a gauge term that enforces the incompressibility condition for the vector potential $A$. We note that for strong solutions of the ideal MHD equations (with divergence-free magnetic field), the formulations \eqref{MHD1}-\eqref{MHD3} respectively \eqref{mhdpotential1}-\eqref{mhdpotential3} are equivalent.

In particular, we require that $\nabla \cdot B \lvert_{t= 0} = 0$. We recall that the divergence-free constraint for the magnetic field describes the absence of magnetic monopoles \cite{goedbloed}. For classical solutions, the divergence-free property of the magnetic field is preserved by the evolution given in equation \eqref{MHD2}. However, this need not be true for weak solutions, as $\Div B$ satisfies a transport equation where the vector field has low regularity (and therefore uniqueness of weak solutions of the transport equation does not hold in general). We will address this matter in more detail in Section \ref{divergencesection}. 

We recall that the MHD system has three formally conserved quantities (under the assumption $\nabla \cdot B = 0$), namely:
\begin{enumerate}
    \item The total energy
    \begin{equation}
    \mathcal{E} (t) = \int_{\mathbb{T}^3} \big( \lvert u(x,t) \rvert^2 + \lvert B(x,t) \rvert^2 \big) dx.
    \end{equation}
    \item The cross helicity
    \begin{equation}
    \mathcal{H}_{u,B} (t) = \int_{\mathbb{T}^3} u(x,t) \cdot B(x,t) dx.
    \end{equation}
    \item The magnetic helicity
    \begin{equation} \label{magnetichelicity}
    \mathcal{H}_{B,B} = \int_{\mathbb{T}^3} A(x,t) \cdot B(x,t) dx. 
    \end{equation}
\end{enumerate}
We observe that the magnetic helicity is independent of the choice of the vector potential $A$. The inclusion of an additional gradient term $\nabla \varphi$ in the vector potential leads to the following extra term in the magnetic helicity (when $\varphi$ is suitably smooth)
\begin{equation*}
\int_{\mathbb{T}^3} \nabla \varphi (x,t) \cdot B(x,t) dx = 0,
\end{equation*}
because of the divergence-free property of the magnetic field.

One can consider analogues of the Onsager conjecture for each of these three conserved quantities. In this work, we will only consider the magnetic helicity. As was said before, the derivation of these conservation laws relies on the magnetic field being divergence-free. For the sufficient condition for magnetic helicity conservation we derive in this paper, we will assume that $B$ is weakly divergence-free.

In \cite{caflisch} it was shown that under sufficient Besov regularity, weak solutions of the ideal MHD system conserve energy, these results were subsequently refined by \cite{kang} using the method of \cite{cheskidovconstantin} (see also \cite{yu}).  Sufficient conditions for the conservation of magnetic helicity were found in \cite{caflisch,kang} (see also \cite{aluie,faracobounded}). Sufficient conditions for the conservation of cross helicity were established in \cite{kang,yu}. In \cite{faracotaylor} Taylor's conjecture was proven, which states that weak-$*$ limits of Leray-Hopf weak solutions of the viscous and resistive MHD equations conserve magnetic helicity (see also \cite{faracobounded}). For each of these results it was assumed but not established that $\nabla \cdot B = 0$ in a weak sense.

Using the method of convex integration, examples of weak solutions that violate the three conservation laws were constructed in \cite{beekie}. In \cite{faracobounded} weak solutions were found which conserve magnetic helicity, but did not conserve energy and cross helicity. Further results on the nonuniqueness of weak solutions can be found in \cite{limhd,linonuniqueness}.

We now give an overview of the results in this paper. In Section \ref{eulersection} we consider the Euler equations. We establish an equation of local helicity balance under some regularity assumptions. This balance contains a defect measure, which captures the potential lack of regularity of the solution as well as the helicity flux. 

We then proceed to prove a sufficient condition for the defect measure to be zero and hence the helicity to be conserved. Subsequently we prove that several of the sufficient criteria for helicity conservation in the literature imply that the sufficient condition introduced in this paper is satisfied, and therefore this new condition is stronger.

In Section \ref{inviscidlimitsection} we prove a sufficient condition for the helicity to be conserved in the zero viscosity limit (of the Navier-Stokes equations). In Section \ref{triplecorrsection} we establish a relation between a third-order structure function and the defect measure found in Section \ref{eulersection}. 

In Section \ref{MHDsection} the MHD equations are studied, in particular a new sufficient condition for conservation of the magnetic helicity is obtained. These results are extended to the kinematic dynamo model in Section \ref{dynamosection}. In Section \ref{divergencesection}, we prove the global existence of Leray-Hopf weak solutions for the viscous and resistive MHD equations, for general $L^2 (\mathbb{T}^3)$ initial data for the magnetic field (which need not be divergence-free). 

We then show that for such Leray-Hopf weak solutions, if the magnetic field is divergence-free initially, it will remain divergence-free. Finally, we show that weak solutions of the ideal MHD equations arising as weak-$*$ vanishing viscosity and resistivity limits of Leray-Hopf solutions with divergence-free initial data, will also have the property that the magnetic field will remain divergence-free if it is initially divergence-free. We note that the preservation of the divergence-free condition for the magnetic field could serve as a (partial) selection/ruling out criterion for physically relevant weak solutions of the ideal MHD system (in a similar spirit as in \cite{bardos2012,bardos2013,szekelyhidi2011}), and the results of this paper show that this criterion is satisfied in the vanishing viscosity limit.

Some concluding remarks are provided in Section \ref{conclusion}. In Appendix \ref{paradifferentialappendix} some inequalities from paradifferential calculus are stated that will be used in this paper, for the convenience of the reader. For some inequalities we will provide a proof for the sake of completeness, as we could not locate them in the literature in the required form.

In this paper we will use the Einstein summation convention. We will often use Besov spaces, which are denoted by $B^s_{p,q} (\mathbb{T}^3)$. Moreover, throughout we will use the convention that $f \in H^{s+} (\mathbb{T}^3)$ means that there exists an $\epsilon > 0$ such that $f \in H^{s + \epsilon} (\mathbb{T}^3)$. We will denote spatial duality brackets by $\langle \cdot, \cdot \rangle$ (so without subscript), while temporal and spatial duality brackets will be denoted by $\langle \cdot , \cdot \rangle_{x,t}$. Finally, $C$ will denote a numerical constant which may change value from line to line. 

\section{The incompressible Euler equations} \label{eulersection}
We start by introducing a new weak formulation of the Euler equations in the vorticity form, such solutions will be referred to as functional vorticity solutions. We introduce this new terminology to contrast it with the usual distributional solutions of the vorticity equation, as in this case the vorticity does not need to lie in an $L^p (\mathbb{T}^3)$ space but is assumed to have negative Sobolev regularity.
\begin{definition} \label{functionalvorticity}
We call a pair $(u, \omega)$ a functional vorticity solution of the Euler equations \eqref{vorticityequation} (or \eqref{eulerequation}) if $u \in L^\infty ((0,T) ; H^{1/2+} (\mathbb{T}^3))$ and it satisfies the following equation for all $\psi \in \mathcal{D} ( \mathbb{T}^3 \times (0,T))$
\begin{align} \label{functionalvorticityeq}
&\int_0^T \langle \omega, \partial_t \psi \rangle dt + \int_0^T \big\langle u \otimes \omega - \omega \otimes u , \nabla \psi \big\rangle dt = 0. 
\end{align}
Here $\langle \cdot, \cdot \rangle$ denote the spatial distributional duality brackets. Moreover, $\omega = \nabla \times u$ and $\nabla \cdot u = \nabla \cdot \omega = 0$ in the distributional sense. 
\end{definition}
\begin{remark}
The notion of a functional vorticity solution is well-defined because we can interpret the duality brackets as duality brackets between $B^{-1/2}_{1,2} (\mathbb{T}^3)$ and $B^{1/2+\theta}_{\infty,2} (\mathbb{T}^3)$ for some small $\theta > 0$. This is done by using Bony's paradifferential calculus from which it follows that $u \otimes \omega \in L^\infty ((0,T); B^{-1/2}_{1,2} (\mathbb{T}^3))$, see Lemma \ref{paraproductlemma} for more details. \\
We note that in order to obtain a well-defined concept of solution, it would have been sufficient to require $u \in L^2 ((0,T); H^{1/2+} (\mathbb{T}^3))$. However, such regularity is not sufficient to make sense of the (global) helicity pointwise (almost everywhere) in time, which is needed to prove a sufficient condition for helicity conservation. Moreover, observe that we have not specified how the initial data is attained for functional vorticity solutions, but this can be done by requiring the attainment of the initial data (for $u$) in $H^{1/2+} (\mathbb{T}^3)$ in the strong sense.
\end{remark}
We recall that the (standard) weak formulation of the Euler equations is given by (for divergence-free test functions $\zeta \in \mathcal{D} (\mathbb{T}^3 \times (0,T);\mathbb{R}^3)$)
\begin{equation} \label{weakformulationeuler}
\int_0^T \int_{\mathbb{T}^3} u \partial_t \zeta dx dt + \int_0^T \int_{\mathbb{T}^3} \big[ (u \otimes u) : \nabla \zeta \big] dx dt = 0,
\end{equation}
where $u \in L^\infty ((0,T); L^2 (\mathbb{T}^3))$. 
We first show that if $u \in L^\infty ((0,T); H^{1/2+} (\mathbb{T}^3))$ then the definition of a functional vorticity solution is equivalent to the standard definition of a weak solution to the Euler equations. Before we prove the equivalence, let us first fix some notation. We let $\phi : \mathbb{R}^3 \rightarrow \mathbb{R}$ be a standard radial $C^\infty$ compactly supported mollifier such that $\int_{\mathbb{R}^3} \phi dx = 1$. We then define
\begin{equation}
\phi_\epsilon (x) \coloneqq \frac{1}{\epsilon^3} \phi \bigg( \frac{x}{\epsilon} \bigg).
\end{equation}
Moreover, we introduce the notation $u^\epsilon \coloneqq u * \phi_\epsilon$. 
\begin{proposition} \label{funcvorticityprop}
Let $u \in L^\infty ((0,T); H^{1/2+} (\mathbb{T}^3))$ be a weak solution of the incompressible Euler equations \eqref{eulerequation}. Then $u$ and its associated vorticity $\omega \in L^\infty ((0,T); H^{-1/2+} (\mathbb{T}^3))$ satisfy the weak formulation in equation \eqref{functionalvorticityeq}, i.e. $(u, \omega)$ is a functional vorticity solution.
\end{proposition}
\begin{proof}
In equation \eqref{weakformulationeuler}, we choose $\zeta = \nabla \times \psi$, where $\psi$ is a divergence-free test function in $\mathcal{D} (\mathbb{T}^3 \times (0,T); \mathbb{R}^3)$. Let $\epsilon_{ijk}$ be the Levi-Civita tensor. We can then write the weak formulation from \eqref{weakformulationeuler} as follows in component notation (recall that we are using the Einstein summation convention)
\begin{equation*}
\int_0^T \int_{\mathbb{T}^3} \bigg[ u_i \partial_t \big( \epsilon_{ijk} \partial_j \psi_k \big) + u_i u_j \partial_i \big( \epsilon_{jkl} \partial_k \psi_l) \bigg] dx dt = 0.
\end{equation*}
By definition of the distributional derivative the left-hand side is equivalent to equation \eqref{functionalvorticityeq}. This can be established rigorously by introducing a mollified version of the left-hand side, which is
\begin{equation} \label{mollifiedformulation}
\int_0^T \int_{\mathbb{T}^3} \bigg[ u_i^\epsilon \partial_t \big( \epsilon_{ijk} \partial_j \psi_k \big) + u_i^\epsilon u_j^\epsilon \partial_i \big( \epsilon_{jkl} \partial_k \psi_l) + \big[ (u_i u_j)^\epsilon - u_i^\epsilon u_j^\epsilon \big] \partial_i \big( \epsilon_{jkl} \partial_k \psi_l) \bigg] dx dt = 0.
\end{equation}
It is clear that the commutator term $\big[ (u_i u_j)^\epsilon - u_i^\epsilon u_j^\epsilon \big]$ vanishes as $\epsilon \rightarrow 0$. By inequality \eqref{paradiffineq3} it follows that $u \otimes u \in L^\infty ((0,T); W^{1/2,1} (\mathbb{T}^3))$, because $u \in L^\infty ((0,T); H^{1/2+} (\mathbb{T}^3))$. Therefore by applying Proposition 17.12 in \cite{leoni}, we find (for $i,j=1,2,3$)
\begin{align*}
\lVert (u_i u_j)^\epsilon - u_i^\epsilon u_j^\epsilon \rVert_{L^\infty (L^1)} &\lesssim \lVert (u_i u_j)^\epsilon - u_i u_j \rVert_{L^\infty (W^{1/2,1})} + \lVert u_i u_j - u_i^\epsilon u_j^\epsilon \rVert_{L^\infty (L^2)} \\
&\lesssim \lVert (u_i u_j)^\epsilon - u_i u_j \rVert_{L^\infty (W^{1/2,1})} + 2 \lVert u \rVert_{L^\infty (H^{1/2+})} \lVert u - u^\epsilon \rVert_{L^\infty (H^{1/2+})} \xrightarrow[]{\epsilon \rightarrow 0} 0.
\end{align*}
Now let $\omega^\epsilon \coloneqq \nabla \times u^\epsilon$, by standard identities from vector calculus it is clear that the weak formulation \eqref{mollifiedformulation} can be written as
\begin{equation*}
\int_0^T \int_{\mathbb{T}^3} \bigg[ \omega^\epsilon \partial_t \psi + u^\epsilon \otimes \omega^\epsilon : \nabla \psi - \omega^\epsilon \otimes u^\epsilon : \nabla \psi + \big[ (u_i u_j)^\epsilon - u_i^\epsilon u_j^\epsilon \big] \partial_i \big( \epsilon_{jkl} \partial_k \psi_l \big) \bigg] dx dt = 0.
\end{equation*}
One then observes that at the expense of an additional commutator term (by using that $\omega^\epsilon$ is divergence-free), in fact the equation holds for all test functions (which are not necessarily divergence-free). By defining a test function $\Psi \coloneqq \psi + \nabla \varphi$ (where $\varphi \in \mathcal{D} (\mathbb{T}^3 \times (0,T))$) we observe that
\begin{align}
&\int_0^T \int_{\mathbb{T}^3} \bigg[ \omega^\epsilon \partial_t \Psi + u^\epsilon \otimes \omega^\epsilon : \nabla \Psi - \omega^\epsilon \otimes u^\epsilon : \nabla \Psi + \big[ (u_i u_j)^\epsilon - u_i^\epsilon u_j^\epsilon \big] \partial_i \big( \epsilon_{jkl} \partial_k \psi_l + \partial_j \varphi \big) \bigg] dx dt \nonumber \\
&= \int_0^T \int_{\mathbb{T}^3} \big[ (u_i u_j)^\epsilon - u_i^\epsilon u_j^\epsilon \big] \partial_i \partial_j \varphi dx dt. \label{mollifiedformulation2}
\end{align}
We recall that due to the Helmholtz-Weyl decomposition (see for example \cite[~Theorem 2.6]{robinsonbook}) the weak formulation therefore holds for all test functions $\Psi \in \mathcal{D} (\mathbb{T}^3 \times (0,T))$ (which are not necessarily divergence-free). The last thing that remains to be shown is that $u^\epsilon \otimes \omega^\epsilon \xrightarrow[]{\epsilon \rightarrow} u \otimes \omega$ in $L^1 ((0,T); B^{-1/2}_{1,1} (\mathbb{T}^3))$. By applying inequality \eqref{paradiffineq1} we find
\begin{align*}
&\lVert u^\epsilon \otimes \omega^\epsilon - u \otimes \omega \rVert_{L^1 (B^{-1/2}_{1,1})} \leq \lVert u^\epsilon \otimes \omega^\epsilon - u^\epsilon \otimes \omega + u^\epsilon \otimes \omega - u \otimes \omega \rVert_{L^1 (B^{-1/2}_{1,1})} \\
&\lesssim \lVert u^\epsilon \rVert_{L^2 (H^{1/2+})} \lVert \omega^\epsilon - \omega \rVert_{L^2 (H^{-1/2+})} + \lVert u^\epsilon - u \rVert_{L^2 (H^{1/2+})} \lVert \omega \rVert_{L^2 (H^{-1/2+})} \xrightarrow[]{\epsilon \rightarrow 0} 0,
\end{align*}
where we have applied Lemma \ref{mollificationbesovlemma} and Proposition 17.12 in \cite{leoni}. Therefore we conclude that equation \eqref{mollifiedformulation2} converges to equation \eqref{functionalvorticityeq} for all test functions in $\mathcal{D} (\mathbb{T}^3 \times (0,T))$.
\end{proof}
Next we show that a functional vorticity solution also solves the Euler equations in the velocity formulation. 
\begin{proposition}
Let $(u, \omega)$ be a functional vorticity solution of the Euler equations, then $u$ is a weak solution of the Euler equations in the standard sense (i.e., $u$ satisfies equation \eqref{weakformulationeuler}).
\end{proposition}
\begin{proof}
We start by rewriting equation \eqref{functionalvorticityeq} as follows (because $\nabla \times u = \omega$ as distributions)
\begin{align} \label{curlformvorticityeq}
&\int_0^T \langle \nabla \times u, \partial_t \psi \rangle dx dt - \int_0^T \big\langle \omega \times u , \nabla \times \psi \big\rangle dt = 0,
\end{align}
where we have used the following identity
\begin{equation} \label{vectorcalculusidentity}
\nabla \times ( \omega \times u ) = \nabla \cdot (u \otimes \omega - \omega \otimes u).
\end{equation}
Notice that this identity holds in the space $L^\infty ((0,T); B^{-3/2}_{1,1} (\mathbb{T}^3))$. One can prove it as follows: Let $u^\epsilon$ be the mollification of $u$, and define $\omega^\epsilon \coloneqq \nabla \times u^\epsilon$. By standard identities from vector calculus, it follows that equation \eqref{vectorcalculusidentity} holds for $u^\epsilon$ and $\omega^\epsilon$. One can deduce that by using Lemma \ref{paraproductlemma} (similarly to the proof of Proposition \ref{funcvorticityprop})
\begin{align*}
&\lVert \nabla \times ( \omega \times u ) - \nabla \times ( \omega^\epsilon \times u^\epsilon ) \rVert_{L^1 (B^{-3/2}_{1,1})} \lesssim \lVert \omega \times u - \omega^\epsilon \times u^\epsilon \rVert_{L^1 (B^{-1/2}_{1,1})} \\
&\lesssim \lVert \omega \rVert_{L^2 (H^{-1/2+})} \lVert u - u^\epsilon \rVert_{L^2 (H^{1/2+})} + \lVert \omega^\epsilon - \omega \rVert_{L^2 (H^{-1/2+})} \lVert u^\epsilon \rVert_{L^2 (H^{1/2+})} \xrightarrow[]{\epsilon \rightarrow 0} 0,
\end{align*}
where the convergence to zero follows due to Lemma \ref{mollificationbesovlemma} and Proposition 17.12 in \cite{leoni}. Therefore identity \eqref{vectorcalculusidentity} holds. Subsequently, we use the related identity
\begin{equation} \label{vectorcalculusidentity2}
\omega \times u = \nabla \cdot (u \otimes u) - \frac{1}{2} \nabla (\lvert u \rvert^2),
\end{equation}
which holds in $L^\infty ((0,T); B^{-1/2}_{1,1} (\mathbb{T}^3))$. Again this can be proved by observing that this identity holds for $u^\epsilon$ and $\omega^\epsilon = \nabla \times u^\epsilon$ and using that $\omega^\epsilon \times u^\epsilon \xrightarrow[]{\epsilon \rightarrow 0} \omega \times u$ in $L^1 ((0,T); B^{-1/2}_{1,1} (\mathbb{T}^3))$ and $u^\epsilon \otimes u^\epsilon \xrightarrow[]{\epsilon \rightarrow 0} u \otimes u$ in $L^\infty ((0,T); L^1 (\mathbb{T}^3))$ (as was shown before). By identity \eqref{vectorcalculusidentity2} the second integral in equation \eqref{curlformvorticityeq} can be written as follows
\begin{align*}
- \int_0^T \big\langle \omega \times u , \nabla \times \psi \big\rangle dt &= \int_0^T \big\langle u \otimes u , \nabla ( \nabla \times \psi) \big\rangle dt.
\end{align*}
We observe that the term $\frac{1}{2} \nabla (\lvert u \rvert^2)$ does not appear, as it vanishes due to $\nabla \times \psi$ being divergence-free. If we set $\zeta = \nabla \times \psi$, we can rewrite equation \eqref{curlformvorticityeq} as
\begin{equation*}
\int_0^T \int_{\mathbb{T}^3} u \partial_t \zeta dx dt + \int_0^T \big[ (u \otimes u) : \nabla \zeta \big] dt = 0.
\end{equation*}
We observe that this equation holds for all divergence-free test functions $\zeta \in \mathcal{D} (\mathbb{T}^3 \times (0,T);\mathbb{R}^3)$ because divergence-free vector fields can be written as the curl of another vector field by means of the Biot-Savart law (see for example \cite[~Theorem 12.2]{robinsonbook}). Hence we conclude that equation \eqref{weakformulationeuler} holds.
\end{proof}
We also remark that for functional vorticity solutions the pressure can be recovered in the standard fashion, as we have just proven such solutions to also be weak solutions of the Euler equations for the space of divergence-free test functions. Then by applying the De Rham theorem (see \cite[~Chapter 1, Proposition 1.1]{temambook}) one finds that there exist a pressure $p \in \mathcal{D}' (\mathbb{T}^3 \times (0,T))$ such that
\begin{equation*}
\nabla p = - \partial_t u - \nabla \cdot (u \otimes u) \quad \text{in } \mathcal{D}' (\mathbb{T}^3 \times (0,T)).
\end{equation*}
Then one can check that the pressure $p$ satisfies the following usual elliptic problem 
\begin{equation} \label{ellipticproblem}
-\Delta p = (\nabla \otimes \nabla) : (u \otimes u), \quad \int_{\mathbb{T}^3} p dx = 0.
\end{equation}
By applying standard results from elliptic regularity theory, for functional vorticity solutions we have at least $p \in L^\infty ((0,T); W^{1/2,1} (\mathbb{T}^3))$, as $u \otimes u \in L^1 ((0,T); W^{1/2,1+} (\mathbb{T}^3)$ by Lemma \ref{paraproductlemma}.

From now on, we will often refer to functional vorticity solutions without mentioning Definition \ref{functionalvorticity}, for the sake of brevity. Now we derive the equation of local helicity balance. In order to do so, we will first need a proposition.
\begin{proposition} \label{gentestfunctions}
The weak formulation for functional vorticity solutions to the Euler equations, given in equation \eqref{functionalvorticityeq}, still holds for test functions $\psi \in W^{1,1}_0 ((0,T); H^{1/2+} (\mathbb{T}^3)) \cap L^1 ((0,T); H^4 (\mathbb{T}^3))$.
\end{proposition}
\begin{proof}
For any $\psi \in W^{1,1}_0 ((0,T); H^{1/2+} (\mathbb{T}^3)) \cap L^1 ((0,T); H^4 (\mathbb{T}^3))$ we take an approximating sequence $\{ \psi_m \} \subset \mathcal{D} (\mathbb{T}^3 \times (0,T);\mathbb{R}^3)$ such that $\psi_m \rightarrow \psi$ in $W^{1,1}_0 ((0,T); H^{1/2+} (\mathbb{T}^3)) \cap L^1 ((0,T); H^4 (\mathbb{T}^3))$. Then it is easy to check that 
\begin{align*}
&\int_0^T \langle \omega, \partial_t \psi_m \rangle dx dt + \int_0^T \big\langle u \otimes \omega - \omega \otimes u , \nabla \psi_m \big\rangle dt \xrightarrow[]{m \rightarrow \infty} \int_0^T \langle \omega, \partial_t \psi \rangle dt \\
&+ \int_0^T \big\langle u \otimes \omega - \omega \otimes u , \nabla \psi \big\rangle dt,
\end{align*} 
as $u \otimes \omega \in L^\infty ((0,T);B^{-1/2}_{1,1} (\mathbb{T}^3))$.
\end{proof}
We now state the same proposition for the Euler equations in velocity form.
\begin{proposition} \label{gentestfunctions2}
The weak formulation of the Euler equations in velocity form given in equation \eqref{weakformulationeuler} still holds for test functions in $W^{1,1}_0 ((0,T); L^2 (\mathbb{T}^3)) \cap L^1 ((0,T); H^3 (\mathbb{T}^3))$.
\end{proposition}
\begin{proof}
The proof proceeds in the same way as the proof of Proposition \ref{gentestfunctions}.
\end{proof}
We next introduce the following notation (where we recall that $u \in L^\infty ((0,T); H^{1/2+} (\mathbb{T}^3))$)
\begin{equation}
\delta u (\xi;x,t) \coloneqq u ( x + \xi, t) - u (x,t), \quad \text{for } x, \xi \in \mathbb{T}^3. 
\end{equation}
As $\omega \in L^\infty ((0,T); H^{-1/2+} (\mathbb{T}^3))$ and hence is only a functional, we define $\delta \omega$ by duality. This means that for $\xi \in \mathbb{T}^3$, $t \in (0,T)$ and $\psi \in L^\infty ((0,T); H^{1/2} (\mathbb{T}^3))$
\begin{equation} \label{dualityincrement}
\langle \delta \omega (\xi; \cdot, t) , \psi (\cdot, t) \rangle_{H^{-1/2} (\mathbb{T}^3) \times H^{1/2} (\mathbb{T}^3)} = \langle \omega (\cdot, t) , \delta \psi (-\xi; \cdot, t) \rangle_{H^{-1/2} (\mathbb{T}^3) \times H^{1/2} (\mathbb{T}^3)}.
\end{equation}
Now we turn to establishing the equation of local helicity balance. 
\begin{theorem} \label{helicitybalance}
Let $(u,\omega)$ be a functional vorticity solution of the Euler equations such that $u \in L^3 ((0,T); W^{1/2+,3} (\mathbb{T}^3))$. Then the following equation of local helicity balance holds for all test functions $\varphi \in \mathcal{D} (\mathbb{T}^3 \times (0,T); \mathbb{R})$
\begin{align}
&\int_0^T \bigg[ - \bigg\langle 2 \omega u, \partial_t \varphi \bigg\rangle + \bigg\langle  2 \nabla \cdot (p \omega) + \nabla \cdot \big[ 2 (u \cdot \omega) u - \lvert u \rvert^2 \omega \big], \varphi \bigg\rangle \bigg] dt \nonumber \\
&+ \bigg\langle D_{1} (u, \omega) - \frac{1}{2} D_{2} (u, \omega) ,  \varphi \bigg\rangle_{x,t} = 0. \label{localhelicitybalance}
\end{align}
where the defect terms are given by
\begin{align}
D_1 (u,\omega) - \frac{1}{2} D_2 (u, \omega) &\coloneqq \lim_{\epsilon \rightarrow 0} \bigg[\int_{\mathbb{R}^3} \nabla \phi_\epsilon (\xi) \cdot \delta u (\xi ; x,t) (\delta \omega (\xi ;x,t) \cdot \delta u (\xi ; x, t) ) d \xi \nonumber \\
&- \frac{1}{2} \int_{\mathbb{R}^3} \nabla \phi_\epsilon (\xi) \cdot \delta \omega (\xi ; x,t) \lvert \delta u (\xi ;x,t) \rvert^2 d \xi \bigg],
\end{align}
where the convergence is in the sense of distributions and independent of the choice of the mollifier. Specifically, we have the following evolution equation for the helicity density
\begin{equation}
2 \partial_t (u \omega) + 2 \nabla \cdot (p \omega) + \nabla \cdot \big[2 (u \cdot \omega) u - \lvert u \rvert^2 \omega \big] + \bigg[ D_1 (u,\omega) - \frac{1}{2} D_2 (u, \omega) \bigg] = 0, \quad \text{in } \mathcal{D}' (\mathbb{T}^3 \times (0,T)).
\end{equation}
\end{theorem}
\begin{proof}
We first mollify the velocity and vorticity formulations of the Euler equations, i.e. equations \eqref{eulerequation} and \eqref{vorticityequation} respectively, to give that
\begin{align}
&\partial_t u^\epsilon + \nabla \cdot ( u \otimes u)^\epsilon + \nabla p^\epsilon = 0, \label{mollifiedeuler1} \\
&\partial_t \omega^\epsilon + \nabla \cdot (u \otimes \omega)^\epsilon - \nabla \cdot (\omega \otimes u)^\epsilon = 0. \label{mollifiedeuler2}
\end{align}
We know that $u^\epsilon, \omega^\epsilon \in L^\infty ((0,T); C^\infty (\mathbb{T}^3))$ and from the mollified equations one can see that $\partial_t u^\epsilon, \partial_t \omega^\epsilon \in L^\infty ((0,T); C^\infty (\mathbb{T}^3))$. Hence we are able to conclude that $u^\epsilon, \omega^\epsilon \in W^{1,\infty} ((0,T); \linebreak C^\infty (\mathbb{T}^3))$. By Propositions \ref{gentestfunctions} and \ref{gentestfunctions2} it follows that we are allowed to use $u^\epsilon$ and $\omega^\epsilon$ in the weak formulations given in equations \eqref{functionalvorticityeq} and \eqref{weakformulationeuler}.

Due to Proposition \ref{funcvorticityprop} functional vorticity solutions satisfy the weak formulation \eqref{weakformulationeuler} (which is for divergence-free test functions). As has been explained before, the pressure can then be recovered by means of the De Rham theorem (see \cite[~Chapter 1, Proposition 1.1]{temambook}) and one finds it satisfies the elliptic problem \eqref{ellipticproblem}. As $u \otimes u \in L^{3/2} ((0,T); W^{1/2+,3/2} (\mathbb{T}^3))$ by Lemma \ref{paraproductlemma}, we have that $p \in L^{3/2} ((0,T); W^{1/2+,3/2} (\mathbb{T}^3))$ by results from elliptic regularity.

Now let $\varphi \in \mathcal{D} (\mathbb{T}^3 \times (0,T); \mathbb{R})$. We take $\varphi u^\epsilon$ as a test function in equation \eqref{functionalvorticityeq} and $\varphi \omega^\epsilon$ as a test function in equation \eqref{weakformulationeuler}. Adding them together gives
\begin{align}
&\int_0^T \int_{\mathbb{T}^3} \bigg[ u \partial_t (\omega^\epsilon \varphi) + (u \otimes u) : \nabla (\omega^\epsilon \varphi) + p \nabla \cdot (\omega^\epsilon \varphi) \bigg] dx dt \nonumber \\
&+ \int_0^T \langle \omega, \partial_t (u^\epsilon \varphi) \rangle + \langle u \otimes \omega - \omega \otimes u, \nabla (u^\epsilon \varphi ) \rangle dt = 0. \label{mollifiedtestfunction}
\end{align}
Now by taking the $H^{-1/2} (\mathbb{T}^3) \times H^{1/2} (\mathbb{T}^3)$ duality bracket respectively the $L^2 (\mathbb{T}^3)$ inner product of equations \eqref{mollifiedeuler1} and \eqref{mollifiedeuler2} with $\omega \varphi$ (which lies in $L^\infty ((0,T); H^{-1/2+} (\mathbb{T}^3))$) and respectively $u \varphi$ then subtracting it from equation \eqref{mollifiedtestfunction} we obtain
\begin{align*}
&\int_0^T \int_{\mathbb{T}^3} \bigg[ u \partial_t (\omega^\epsilon \varphi) + (u \otimes u) : \nabla (\omega^\epsilon \varphi) + p \nabla \cdot (\omega^\epsilon \varphi) - u \varphi \partial_t \omega^\epsilon -  u \varphi \cdot \nabla \cdot (u \otimes \omega)^\epsilon \\
&+  u \varphi \cdot \nabla \cdot (\omega \otimes u)^\epsilon \bigg] dx dt + \int_0^T \bigg[ \langle \omega, \partial_t (u^\epsilon \varphi) \rangle + \langle u \otimes \omega - \omega \otimes u, \nabla (u^\epsilon \varphi ) \rangle  \\
&- \langle \omega \varphi, \partial_t u^\epsilon + \nabla \cdot ( u \otimes u)^\epsilon + \nabla p^\epsilon \rangle \bigg] dt= 0.
\end{align*}
We first observe that the time derivative terms can be written as follows
\begin{equation*}
\int_0^T \bigg[ \int_{\mathbb{T}^3} \big(u \partial_t (\omega^\epsilon \varphi) - u \varphi \partial_t \omega^\epsilon \big) dx + \langle \omega, \partial_t (u^\epsilon \varphi) \rangle - \langle \omega \varphi, \partial_t u^\epsilon \rangle \bigg] dt = \int_0^T \langle \omega^\epsilon u + \omega u^\epsilon, \partial_t \varphi \rangle dt.
\end{equation*}
Now we write the pressure terms as follows
\begin{align*}
\int_0^T \bigg[ \int_{\mathbb{T}^3} p \nabla \cdot (\omega^\epsilon \varphi) dx - \langle \omega \varphi, \nabla p^\epsilon \rangle \bigg] dt = \int_0^T \langle p \omega^\epsilon + \omega p^\epsilon, \nabla \varphi \rangle dt.
\end{align*}
To deal with the advective terms, we introduce two defect terms given by
\begin{align}
&D_{1,\epsilon} (u, \omega) (x,t) = \int_{\mathbb{R}^3} \nabla \phi_\epsilon (\xi) \cdot \delta u (\xi ; x,t) (\delta \omega (\xi ;x,t) \cdot \delta u (\xi ; x, t) ) d \xi = - \nabla \cdot ( (\omega \cdot u) u)^\epsilon \nonumber \\
&+ u \cdot \nabla (\omega \cdot u)^\epsilon + u \cdot \big( \nabla \cdot (u \otimes \omega)^\epsilon \big) + \omega \cdot \big( \nabla \cdot (u \otimes u)^\epsilon \big) - u \otimes \omega : \nabla u^\epsilon - u \otimes u : \nabla \omega^\epsilon, \label{D1epsilon} \\
&D_{2,\epsilon} (u, \omega) (x,t) = \int_{\mathbb{R}^3} \nabla \phi_\epsilon (\xi) \cdot \delta \omega (\xi ; x,t) \lvert \delta u (\xi ;x,t) \rvert^2 d \xi = - \nabla \cdot ( \lvert u \rvert^2 \omega)^\epsilon + \omega \cdot \nabla (\lvert u \rvert^2 )^\epsilon \nonumber \\
&+ 2 u \cdot \big( \nabla \cdot (\omega \otimes u)^\epsilon \big) - 2 \omega \otimes u : \nabla u^\epsilon. \label{D2epsilon}
\end{align}
Observe that equations \eqref{D1epsilon} and \eqref{D2epsilon} hold in the space $L^1 ((0,T);B^{-1/2}_{1,3}(\mathbb{T}^3))$, which can be seen by applying inequalities \eqref{paradiffineq1} and \eqref{paradiffineq2} in the Appendix. Now we are able to rewrite the advective terms as follows
\begingroup
\allowdisplaybreaks
\begin{align*}
&\int_0^T \int_{\mathbb{T}^3} \bigg[ (u \otimes u) : \nabla (\omega^\epsilon \varphi)  -  u \varphi \cdot \big(\nabla \cdot (u \otimes \omega)^\epsilon \big) +  u \varphi \cdot \big(\nabla \cdot (\omega \otimes u)^\epsilon \big) \bigg] dx dt \\
&+ \int_0^T \bigg[ \langle u \otimes \omega - \omega \otimes u, \nabla (u^\epsilon \varphi ) \rangle - \langle \omega \varphi, \nabla \cdot ( u \otimes u)^\epsilon \rangle \bigg] dt \\
&= \int_0^T \bigg[ \bigg\langle (u \cdot \omega^\epsilon) u + (\omega \cdot u^\epsilon) u - (u \cdot u^\epsilon) \omega , \nabla \varphi \bigg\rangle + \frac{1}{2} \bigg\langle 2((\omega \cdot u) u)^\epsilon - 2(\omega \cdot u)^\epsilon u \\
&- ( \lvert u \rvert^2 \omega)^\epsilon + \omega (\lvert u \rvert^2)^\epsilon, \nabla \varphi \bigg\rangle + \bigg\langle - D_{1,\epsilon} (u,\omega) + \frac{1}{2} D_{2,\epsilon} (u,\omega) ,  \varphi \bigg\rangle \bigg] dt.
\end{align*}
\endgroup
Now we can collect the various terms to arrive at the following equation of local helicity balance (for fixed $\epsilon > 0$)
\begin{align}
&\int_0^T \bigg\langle \partial_t \big( \omega^\epsilon u + \omega u^\epsilon \big) + \nabla \cdot (p \omega^\epsilon + p^\epsilon \omega) + \nabla \cdot \big[ (u \cdot \omega^\epsilon) u + (\omega \cdot u^\epsilon) u - (u \cdot u^\epsilon) \omega \big], \varphi \bigg\rangle dt \nonumber \\
&+ \frac{1}{2} \int_0^T \bigg\langle \nabla \cdot \big[ 2((\omega \cdot u) u)^\epsilon - 2(\omega \cdot u)^\epsilon u - ( \lvert u \rvert^2 \omega)^\epsilon + \omega (\lvert u \rvert^2)^\epsilon \big], \varphi \bigg\rangle dt \nonumber \\
&= \int_0^T \bigg\langle - D_{1,\epsilon} (u, \omega) + \frac{1}{2} D_{2,\epsilon} (u, \omega) ,  \varphi \bigg\rangle dt. \label{mollifiedbalance}
\end{align}
We observe that the left-hand side converges in $W^{-1,1} ((0,T); W^{-2,1} (\mathbb{T}^3))$ as $\epsilon \rightarrow 0$, as $\lvert u \rvert^2 \omega \in L^1 ((0,T); B^{-1/2}_{1,1} (\mathbb{T}^3))$. This means that the limit $\lim_{\epsilon \rightarrow 0} \big[ - D_{1,\epsilon} (u, \omega) + \frac{1}{2} D_{2,\epsilon} (u, \omega) \big]$ is well-defined as an element of $W^{-1,1} ((0,T); W^{-2,1} (\mathbb{T}^3))$. We will denote the limit as $-D_{1} (u, \omega) + \frac{1}{2} D_{2} (u, \omega)$, which is a slight abuse of notation as the limits $D_1 (u,\omega)$ and $D_2 (u,\omega)$ might not in general be individually well-defined. Then by taking the distributional limit as $\epsilon \rightarrow 0$, we arrive at the following equation of local helicity balance
\begin{align}
&\int_0^T \bigg[ - \bigg\langle 2 \omega u, \partial_t \varphi \bigg\rangle + \bigg\langle  2 \nabla \cdot (p \omega) + \nabla \cdot \big[ 2 (u \cdot \omega) u - \lvert u \rvert^2 \omega \big], \varphi \bigg\rangle \bigg] dt \nonumber \\
&= \bigg\langle - D_{1} (u, \omega) + \frac{1}{2} D_{2} (u, \omega) ,  \varphi \bigg\rangle_{x,t}. \label{limithelicitybalance}
\end{align}
\end{proof}
We now prove a sufficient condition for the vanishing of the defect terms, which will then imply conservation of the (local) helicity density. We first prove a proposition which will be needed in the proof of the sufficient condition.
\begin{proposition} \label{defectidentitylemma}
Let $\epsilon > 0$ and let $(u, \omega)$ be a functional vorticity solution of the Euler equations such that $u \in L^3 ((0,T); W^{1/2+,3} (\mathbb{T}^3))$. Then the following identity holds
\begin{align}
&\int_0^T \bigg\langle - D_{1,\epsilon} (u, \omega) + \frac{1}{2} D_{2,\epsilon} (u, \omega) , \varphi \bigg\rangle_{B^{-1/2}_{1,3} \times B^{1/2}_{\infty,3/2}} dt \nonumber \\
&= \int_0^T \int_{\mathbb{R}^3} \nabla \phi_\epsilon (\xi) \bigg[- \langle \delta u (\delta \omega \cdot \delta u), \varphi \rangle_{W^{-1/2,1} \times W^{1/2,\infty}} + \frac{1}{2} \langle \delta \omega \lvert \delta u \rvert^2 , \varphi \rangle_{W^{-1/2,1} \times W^{1/2,\infty}} \bigg] d \xi dt. \label{defectidentity}
\end{align}
We note that we have restricted the dual space considered in the duality brackets, but this is sufficient for our purposes. 
\end{proposition}
\begin{remark} \label{defectidentityremark}
We observe that the expression on the left-hand side of equation \eqref{defectidentity} is clearly well-defined (for any fixed $\epsilon > 0$), as it was written out explicitly in equations \eqref{D1epsilon} and \eqref{D2epsilon}. The fact that the right-hand side is well-defined can be seen in multiple ways. As was noted before, the difference $\delta \omega$ was defined by duality in equation \eqref{dualityincrement}. Therefore, we have
\begin{align*}
&\langle \delta u (\delta \omega \cdot \delta u), \varphi \rangle = \langle \delta \omega, (\delta u \cdot \varphi) \delta u \rangle \\
&= \bigg\langle \omega, (\delta u (-\xi ; \cdot , t) \cdot \varphi (\cdot - \xi, t)) \delta u (-\xi ; \cdot , t) - (\delta u (\xi ; \cdot, t) \cdot \varphi (\cdot, t)) \delta u (\xi ; \cdot, t) \bigg\rangle \\
&\langle \delta \omega \lvert \delta u \rvert^2 , \varphi \rangle = \langle \delta \omega , \lvert \delta u \rvert^2 \varphi \rangle = \bigg\langle \omega, \lvert \delta u (-\xi ; \cdot, t) \rvert^2 \varphi (\cdot - \xi , t) - \lvert \delta u ( \xi ; \cdot, t) \rvert^2 \varphi (\cdot, t) \bigg\rangle.
\end{align*}
Inserting these expressions into the right-hand side of equation \eqref{defectidentity} yields
\begin{align}
&\int_0^T \int_{\mathbb{R}^3} \nabla \phi_\epsilon (\xi) \bigg[- \bigg\langle \omega, (\delta u (-\xi ; \cdot , t) \cdot \varphi (\cdot - \xi, t)) \delta u (-\xi ; \cdot , t) - (\delta u (\xi ; \cdot, t) \cdot \varphi (\cdot, t)) \delta u (\xi ; \cdot, t) \bigg\rangle \nonumber \\
&+ \frac{1}{2} \bigg\langle \omega, \lvert \delta u (-\xi ; \cdot, t) \rvert^2 \varphi (\cdot - \xi , t) - \lvert \delta u ( \xi ; \cdot, t) \rvert^2 \varphi (\cdot, t) \bigg\rangle \bigg] d \xi dt. \label{dualdiffquotient}
\end{align}
Now one can see that the integral of the duality bracket together with $\nabla \phi_\epsilon$ with respect to $\xi$ is well-defined as follows. One option is to view $\nabla \phi_\epsilon$ as an element of $\mathcal{D}' (\mathbb{R}^3 \times (0,T))$ and then apply the distributional Fubini-Tonelli theorem, which can be found in \cite[~Theorem IV, Section 3, Chapter IV]{schwartz}. Note that the result from \cite{schwartz} does not apply verbatim, as $\delta u \notin \mathcal{D} (\mathbb{T}^3 \times (0,T); \mathbb{R}^3)$. However, the proof can be adapted to suit this particular case. \\
Another option is that the right-hand side of \eqref{defectidentity} can be interpreted by observing that the duality brackets in equation \eqref{dualdiffquotient} are measurable as functions of $\xi$. This can be proven by noting that the translation operator is continuous in $\xi$, and the fact that both $\delta u$ and $\varphi$ have Sobolev regularity. As the duality brackets are bounded uniformly in $\xi$, it follows that the integral in equation \eqref{dualdiffquotient} is well-defined.
\end{remark}
We now turn to the proof of Proposition \ref{defectidentitylemma}.
\begin{proof}
We introduce a mollification of the vorticity (called $\omega^\eta = \omega * \phi_\eta$) and we observe that the following identity holds (by the Fubini-Tonelli theorem)
\begingroup
\allowdisplaybreaks
\begin{align}
&\int_0^T \bigg\langle - D_{1,\epsilon} (u, \omega^\eta) + \frac{1}{2} D_{2,\epsilon} (u, \omega^\eta) , \varphi \bigg\rangle dt = \int_0^T \int_{\mathbb{R}^3} \nabla \phi_\epsilon (\xi) \bigg[- \langle \delta u (\delta \omega^\eta \cdot \delta u), \varphi \rangle_{W^{-1/2,1} \times W^{1/2,\infty}} \nonumber \\
&+ \frac{1}{2} \langle \delta \omega^\eta \lvert \delta u \rvert^2 , \varphi \rangle_{W^{-1/2,1} \times W^{1/2,\infty}} \bigg] d \xi dt. \nonumber
\end{align}
\endgroup
Then by applying standard estimates from the paradifferential calculus as before (see Appendix \ref{paradifferentialappendix}), one obtains
\begin{align}
&\int_0^T \int_{\mathbb{R}^3} \nabla \phi_\epsilon (\xi) \bigg[- \langle \delta u (\delta \omega^\eta \cdot \delta u), \varphi \rangle_{W^{-1/2,1} \times W^{1/2,\infty}} + \frac{1}{2} \langle \delta \omega^\eta \lvert \delta u \rvert^2 , \varphi \rangle_{W^{-1/2,1} \times W^{1/2,\infty}} \bigg] d \xi dt \nonumber \\
&\xrightarrow[]{\eta \rightarrow 0} \int_0^T \int_{\mathbb{R}^3} \nabla \phi_\epsilon (\xi) \bigg[- \langle \delta u (\delta \omega \cdot \delta u), \varphi \rangle_{W^{-1/2,1} \times W^{1/2,\infty}} + \frac{1}{2} \langle \delta \omega \lvert \delta u \rvert^2 , \varphi \rangle_{W^{-1/2,1} \times W^{1/2,\infty}} \bigg] d \xi dt. \label{etalimit1}
\end{align}
Moreover, by using the explicit form of the defect terms $D_{1,\epsilon} - \frac{1}{2} D_{2,\epsilon}$ in equations \eqref{D1epsilon} and \eqref{D2epsilon} one can establish that
\begin{align}
\int_0^T \bigg\langle - D_{1,\epsilon} (u, \omega^\eta) + \frac{1}{2} D_{2,\epsilon} (u, \omega^\eta) , \varphi \bigg\rangle dt \xrightarrow[]{\eta \rightarrow 0 } \int_0^T \bigg\langle - D_{1,\epsilon} (u, \omega) + \frac{1}{2} D_{2,\epsilon} (u, \omega) , \varphi \bigg\rangle dt, \label{etalimit2}
\end{align}
where we note that the convergence rates in limits \eqref{etalimit1} and \eqref{etalimit2} depend on $\epsilon > 0$. This completes the proof.
\end{proof}
Next we state and prove the sufficient condition for the defect terms to be zero.
\begin{proposition} \label{sufficientconditioneuler}
Let $(u, \omega)$ be a functional vorticity solution of the Euler equations which satisfies the following inequalities
\begin{align}
&\bigg\lvert \langle \delta u (\delta \omega \cdot \delta u), \varphi \rangle_{W^{-1/2,1} \times W^{1/2,\infty}} \bigg\rvert \leq C(t) \lvert \xi \rvert \sigma (\lvert \xi \rvert) \lVert \varphi \rVert_{W^{1/2,\infty}}, \label{defectineq1} \\
&\bigg\lvert \langle \delta \omega \lvert \delta u \rvert^2 , \varphi \rangle_{W^{-1/2,1} \times W^{1/2,\infty}} \bigg\rvert \leq C(t) \lvert \xi \rvert \sigma (\lvert \xi \rvert) \lVert \varphi \rVert_{W^{1/2,\infty}}, \label{defectineq2}
\end{align}
where $C \in L^1 (0,T)$ and $\sigma \in L^\infty_{loc} (\mathbb{R})$ with the property that $\sigma (\lvert \xi \rvert)\rightarrow 0$ as $\lvert \xi \rvert \rightarrow 0$. Then the individual limits $D_{1} (u, \omega)$ and $D_2 (u, \omega)$ exist as distributions and are equal to zero. Consequently, this implies that the solution conserves the (local) helicity density. 
\end{proposition}
\begin{proof}

By using Proposition \ref{defectidentitylemma} as well as inequalities \eqref{defectineq1} and \eqref{defectineq2}, we are able to estimate
\begin{align*}
&\bigg\lvert \int_0^T \bigg\langle - D_{1,\epsilon} (u, \omega) + \frac{1}{2} D_{2,\epsilon} (u, \omega) , \varphi \bigg\rangle dt \bigg\rvert \leq \int_0^T C(t) \lVert \varphi \rVert_{W^{1/2,\infty}}  \int_{\mathbb{R}^3} \lvert \nabla \phi_\epsilon (\xi) \rvert \lvert \xi \rvert \sigma (\lvert \xi \rvert) d \xi dt \\
&= \int_0^T \bigg[ C(t) \lVert \varphi \rVert_{W^{1/2,\infty}}  \int_{\mathbb{R}^3} \lvert \nabla \phi (z) \rvert \lvert z \rvert \sigma (\epsilon \lvert z \rvert) d z \bigg] dt \xrightarrow[]{\epsilon \rightarrow 0} 0,
\end{align*}
where in the second line we have made the change of variable $\xi = \epsilon z$ and we have then used the Lebesgue dominated convergence theorem. 
\end{proof}
Finally, we prove the conservation of (global) helicity under the assumptions of Proposition \ref{sufficientconditioneuler}. 
\begin{theorem} \label{conservationeuler}
Let $(u, \omega)$ be a functional vorticity solution of the Euler equations such that $D_1 (u, \omega) = D_2 (u, \omega) = 0$ and $u \in L^3 ((0,T); W^{1/2+,3} (\mathbb{T}^3))$. Then the helicity is conserved, in particular for almost every $t_1, t_2 \in (0,T)$ it holds that
\begin{equation}
\mathcal{H} (t_1) = \mathcal{H} (t_2).
\end{equation}
\end{theorem}
\begin{proof}
Under the assumption that $D_1 (u, \omega) = D_2 (u, \omega) = 0$ (which occurs for example if the conditions of Proposition \ref{sufficientconditioneuler} are satisfied), then by Theorem \ref{helicitybalance} we end up with the following equation of local helicity balance (for $\varphi \in \mathcal{D} (\mathbb{T}^3 \times (0,T); \mathbb{R})$)
\begin{equation*}
\int_0^T \bigg[ - \bigg\langle 2 \omega u, \partial_t \varphi \bigg\rangle + \bigg\langle  2 \nabla \cdot (p \omega) + \nabla \cdot \big[ 2 (u \cdot \omega) u - \lvert u \rvert^2 \omega \big] , \varphi \bigg\rangle \bigg] dt = 0.
\end{equation*}
Then we take $\chi \in C_c^\infty (\mathbb{R}; \mathbb{R})$ to be a standard smooth mollifier such that $\int_{\mathbb{R}} \chi (t) dt = 1$, $\chi_\delta (t) = \delta^{-1} \chi (t/\delta)$ and $\supp \chi \subset (-1,1)$. Then we take the following test function
\begin{equation*}
\varphi (x,t) = \int_0^t \big( \chi_\delta (t' - t_1) - \chi_\delta (t' - t_2) \big) dt'.
\end{equation*}
We observe that $\varphi (t) = 0$ if $t \in (0,t_1 - \delta) \cup (t_2 + \delta)$ (without loss of generality we take $t_1 < t_2$). Moreover, for $\delta$ sufficiently small we have that $\varphi (t) = 1$ if $t \in (t_1 + \delta, t_2 - \delta)$. Now since $\nabla \varphi \equiv 0$, we obtain
\begin{equation*}
\int_0^T \langle \omega, u \rangle_{H^{-1/2} \times H^{1/2}} \; \chi_\delta (t - t_1) dt = \int_0^T \langle \omega, u \rangle_{H^{-1/2} \times H^{1/2}} \; \chi_\delta (t - t_2) dt.
\end{equation*}
Now by applying the Lebesgue differentiation theorem when taking the limit $\delta \rightarrow 0$ we obtain (for almost every $t_1, t_2 \in (0,T)$)
\begin{equation}
\langle \omega (\cdot, t_1) , u (\cdot, t_1) \rangle_{H^{-1/2} \times H^{1/2}} = \langle \omega (\cdot, t_2), u (\cdot, t_2) \rangle_{H^{-1/2} \times H^{1/2}}.
\end{equation}
Note that we have used the $L^\infty$ regularity in time to make sense of the duality bracket pointwise in time. Therefore we conclude that the functional vorticity solution $(u,\omega)$ conserves the (global) helicity.
\end{proof}
Now we will prove that the sufficient condition given in Proposition \ref{sufficientconditioneuler} is implied by the sufficient conditions for (global) helicity conservation given in \cite{cheskidovconstantin,derosa}. 
\begin{proposition} \label{conservationbesov}
Let $(u, \omega)$ be a functional vorticity solution of the Euler equations such that $u \in L^3 ((0,T); B^{2/3+}_{3,\infty} (\mathbb{T}^3))$, then the (global) helicity is conserved.
\end{proposition}
\begin{proof}
Let $u \in L^3 ((0,T); B^{2/3+\theta}_{3,\infty} (\mathbb{T}^3))$ for some $\theta > 0$, then $\omega \in L^3 ((0,T); B^{-1/3+\theta}_{3,\infty} (\mathbb{T}^3))$. By applying inequality \eqref{paradiffineq2} we know that $\lvert \delta u \rvert^2 \in L^{3/2} ((0,T); B^{2/3+\theta}_{3/2,\infty} (\mathbb{T}^3))$, and similarly for $\delta u \otimes \delta u$ (which we will use in the estimate for the limit $D_1 (u,\omega)$). 

We can then obtain the estimates (by using the identification $W^{-1/3,1} (\mathbb{T}^3) = B^{-1/3}_{1,1} (\mathbb{T}^3)$ as well as inequalities \eqref{paradiffineq1} and \eqref{paradiffineq3}) 
\begin{align*}
\bigg\lvert \langle \delta u (\delta \omega \cdot \delta u), \varphi \rangle_{W^{-1/3,1} \times W^{1/3,\infty}} \bigg\rvert &\leq \lVert \delta u (\delta \omega \cdot \delta u) \rVert_{W^{-1/3,1}} \lVert \varphi \rVert_{W^{1/3,\infty}} \\
&\leq \lVert \delta \omega \rVert_{B^{-1/3+\theta}_{3,\infty}} \lVert \delta u \otimes \delta u \rVert_{B^{1/3}_{3/2,\infty}} \lVert \varphi \rVert_{W^{1/3,\infty}} \\
&\leq \lVert \delta \omega \rVert_{B^{-1/3+\theta}_{3,\infty}} \lVert \delta u \rVert_{B^{\theta/2}_{3,\infty}} \lVert \delta u \rVert_{B^{1/3 + \theta/2}_{3,\infty}} \lVert \varphi \rVert_{W^{1/3,\infty}} \\
&\leq \lvert \xi \rvert^{1 + \theta} \lVert \omega \rVert_{B^{-1/3+\theta}_{3,\infty}} \lVert u \rVert_{B^{2/3+\theta}_{3,\infty}}^2 \lVert \varphi \rVert_{W^{1/3,\infty}}, \\
\bigg\lvert \langle \delta \omega \lvert \delta u \rvert^2 , \varphi \rangle_{W^{-1/3,1} \times W^{1/3,\infty}} \bigg\rvert &\leq \lVert \lvert \delta u \rvert^2 \delta \omega \rVert_{W^{-1/3,1}} \lVert \varphi \rVert_{W^{1/3,\infty}} \leq \lVert \delta \omega \rVert_{B^{-1/3+\theta}_{3,\infty}} \lVert \lvert \delta u \rvert^2 \rVert_{B^{1/3}_{3/2,\infty}} \lVert \varphi \rVert_{W^{1/3,\infty}} \\
&\leq \lVert \delta \omega \rVert_{B^{-1/3+\theta}_{3,\infty}} \lVert \delta u \rVert_{B^{\theta/2}_{3,\infty}} \lVert \delta u \rVert_{B^{1/3 + \theta/2}_{3,\infty}} \lVert \varphi \rVert_{W^{1/3,\infty}} \\
&\leq \lvert \xi \rvert^{1 + \theta} \lVert \omega \rVert_{B^{-1/3+\theta}_{3,\infty}} \lVert u \rVert_{B^{2/3+\theta}_{3,\infty}}^2 \lVert \varphi \rVert_{W^{1/3,\infty}},
\end{align*}
where in the last line of both estimates we have used Lemma \ref{diffquotientlemma}.
Hence the conditions of Proposition \ref{sufficientconditioneuler} are satisfied (as we can take $\sigma (\lvert \xi \rvert) \coloneqq \lvert \xi \rvert^\theta$) from which it follows that $D_1 (u,\omega) = D_2 (u,\omega) = 0$, and therefore the (global) helicity conservation holds due to Theorem \ref{conservationeuler}.
\end{proof}
\begin{proposition}
Suppose that $(u,\omega)$ is a functional vorticity solution of the Euler equations and assume that $u \in L^{2r} ((0,T); W^{\beta, 2p} (\mathbb{T}^3))$ and $\omega \in L^k ((0,T); W^{\alpha,q} (\mathbb{T}^3))$, for $0 < \alpha, \beta < 1$ and $1 \leq p, q, r, k \leq \infty$, such that
\begin{equation*}
\frac{1}{p} + \frac{1}{q} = \frac{1}{r} + \frac{1}{k} = 1, \quad 2 \beta + \alpha > 1.
\end{equation*}
Then the functional vorticity solution conserves helicity.
\end{proposition}
\begin{proof}
We will verify the sufficient condition from Proposition \ref{sufficientconditioneuler}. We can estimate that
\begin{align*}
\lvert \langle \delta u (\delta \omega \cdot \delta u), \varphi \rangle_{W^{-1/2,1} \times W^{1/2,\infty}} \rvert &= \bigg\lvert \int_{\mathbb{T}^3} \delta u (\delta \omega \cdot \delta u) \varphi dx \bigg\rvert \leq \lVert \varphi \rVert_{L^\infty} \lVert \delta u \rVert_{L^p}^2 \lVert \delta \omega \rVert_{L^q}  \\
&\leq \lvert \xi \rvert^{2 \beta + \alpha } \lVert \varphi \rVert_{L^\infty} \lVert \delta u \rVert_{W^{\beta,2 p}}^2 \lVert \delta \omega \rVert_{W^{\alpha,q}}.
\end{align*}
We observe that $\lVert \varphi \rVert_{L^\infty} \lVert \delta u \rVert_{W^{\beta,2 p}}^2 \lVert \delta \omega \rVert_{W^{\alpha,q}} \in L^1 (0,T)$ and we can take $\sigma (\lvert \xi \rvert) \coloneqq \lvert \xi \rvert^{2 \beta + \alpha - 1}$. The term $\lvert \langle \delta \omega \lvert \delta u \rvert^2 , \varphi \rangle_{W^{-1/2,1} \times W^{1/2,\infty}} \rvert$ can be estimated similarly. Therefore by Proposition \ref{sufficientconditioneuler} we know that $D_1 (u, \omega) = D_2 (u, \omega) = 0$. Then by Theorem \ref{conservationeuler} we conclude that the (global) helicity is conserved.
\end{proof}
The results in \cite{wangcompressible,wanghelicity} can be derived in a similar fashion, by using suitable interpolation inequalities.

\section{Conservation of helicity in the vanishing viscosity limit} \label{inviscidlimitsection}
In this section we consider the problem of helicity conservation in the vanishing viscosity limit of the Navier-Stokes equations, which are given by
\begin{equation}
\partial_t u - \nu \Delta u + \nabla \cdot (u \otimes u) + \nabla p = 0, \quad \nabla \cdot u = 0,
\end{equation}
where $\nu > 0$ is the viscosity parameter. 

It is still unclear whether helicity is conserved in the inviscid limit as $\nu \rightarrow 0$. The results of several experimental and numerical studies have suggested either conservation or non-conservation of helicity in the vanishing viscosity limit as $\nu \rightarrow 0$, see for example \cite{yaoreview,meng} (and references therein). In this work, we approach this problem from the analytical point of view, as we will provide a sufficient condition for helicity conservation in the inviscid limit, as $\nu \rightarrow 0$. 

The helicity is not a (formally) conserved quantity of solutions of the Navier-Stokes equations. Formally, for sufficiently smooth solutions of the Navier-Stokes equations the derivative of the helicity satisfies
\begin{equation}
\frac{d}{dt} \mathcal{H} (t) = - 2 \nu \int_{\mathbb{T}^3} \bigg( \nabla u : \nabla \omega \bigg) dx. 
\end{equation}
For the integral on the right-hand side to be well-defined (after integrating in time), we would at least have to require that $u \in L^2 ((0,T); H^{3/2} (\mathbb{T}^3))$. We recall that $L^2 ((0,T); \dot{H}^{3/2} (\mathbb{T}^3))$ is a critical norm in the standard scaling of the Navier-Stokes equations in $\mathbb{R}^3$, which is
\begin{equation}
u_\lambda (x,t) = \lambda^{-1} u (\lambda^{-1} x, \lambda^{-2} t). 
\end{equation}
In fact, by a straightforward adaption of the proof of the Prodi-Serrin criteria (see for example \cite{bedrossianbook}), one can show that any Leray-Hopf weak solution of the Navier-Stokes equations for which the $L^2 ((0,T); H^{3/2} (\mathbb{T}^3))$ norm is finite, is in fact a strong solution. If one wishes to formulate an evolution equation for the (local) helicity density of a solution of the Navier-Stokes equations, it would take the form
\begin{equation} \label{viscoushelicitybalance}
\int_0^T \bigg\langle 2 \partial_t \big( \omega u \big) + 4 \nu \nabla u : \nabla \omega + 2 \nabla \cdot (p \omega) + \nabla \cdot \big[ 2 (u \cdot \omega) u - \lvert u \rvert^2 \omega \big] + D_1 (u,\omega) - \frac{1}{2} D_2 (u, \omega), \varphi \bigg\rangle dt = 0.
\end{equation}
We conclude therefore that equation \eqref{viscoushelicitybalance} can only be rigorously justified for strong solutions of the Navier-Stokes equations, as the required regularity to bound the terms coming from the viscous dissipation implies uniqueness among the class of Leray-Hopf weak solutions of the Navier-Stokes equations, by means of the weak-strong uniqueness result originally proved in \cite{serrin}. 

As formal manipulations are justified for strong solutions, we do not need to prove the equation of local helicity balance and in fact the defect terms $D_1 (u, \omega) = D_2 (u, \omega) = 0$ for strong solutions, which follows from Proposition \ref{sufficientconditioneuler}. As a result the following equation (pointwise in space and time) holds for strong solutions of the Navier-Stokes equations on their interval of existence
\begin{equation}
\partial_t \big( \omega u \big) + 2 \nu \nabla u : \nabla \omega + \nabla \cdot (p \omega) + \frac{1}{2} \bigg( \nabla \cdot \big[ 2 (u \cdot \omega) u - \lvert u \rvert^2 \omega \big] \bigg) = 0.
\end{equation}
We now seek to establish sufficient conditions for helicity to be conserved in the inviscid limit of strong solutions of the Navier-Stokes equations. We prove the following theorem, which is related to the results in \cite{drivaseyink}.
\begin{theorem}
Let $\{ u^\nu \}$ be a sequence of strong solutions of the Navier-Stokes equations with corresponding viscosities $\nu$ on a time interval $[0,T]$, such that the sequence is uniformly bounded (in $\nu$, as $\nu \rightarrow 0$) in the space $L^3 ((0,T);B^{\alpha}_{3,\infty} (\mathbb{T}^3)) \cap L^\infty ((0,T); H^{\beta} (\mathbb{T}^3))$, for $\alpha > \frac{2}{3}$ and $\beta > \frac{1}{2}$. Then for any strong limit $u$ of a subsequence $\{ u^{\nu_k} \}$ converging in $L^3 (\mathbb{T}^3 \times (0,T))$, we have that $u^{\nu_k} \rightarrow u$ in $L^3 ((0,T);B^{\alpha-\theta}_{3,\infty} (\mathbb{T}^3)) \cap L^\infty ((0,T); H^{\beta - \theta} (\mathbb{T}^3))$ for any $\theta > 0$ as $\nu_k \rightarrow 0$. Moreover, any such limit $u$ is a functional vorticity solution of the Euler equations (as introduced in Definition \ref{functionalvorticity}) which conserves the (global) helicity. Finally, there exists at least one such strong limit $u$.
\end{theorem}
\begin{proof}
First we prove that there exists at least one subsequence $\{ u^{\nu_k} \}$ converging strongly in $L^3 (\mathbb{T}^3 \times (0,T))$ to a limit $u \in \cap_{\theta > 0} L^3 ((0,T); B^{\alpha-\theta}_{3,\infty} (\mathbb{T}^3)) \cap L^\infty ((0,T); H^{\beta - \theta} (\mathbb{T}^3))$, for all $\theta > 0$. Similarly to \cite{drivaseyink}, one can show that $\frac{d u^{\nu_k}}{d t} \in L^3 ((0,T);B^{-6/3}_{3,\infty} (\mathbb{T}^3))$ (uniformly in $\nu_k$). Then we apply the Aubin-Lions lemma (see for example \cite{robinsonbook,boyer}) in order to obtain the strong compactness of a subsequence $\{ u^{\nu_k} \}$ in $L^3 ((0,T);B^{\alpha - \theta}_{3,\infty} (\mathbb{T}^3)) \cap L^\infty ((0,T); H^{\beta - \theta} (\mathbb{T}^3))$ for all $\theta > 0$ and in particular in $L^3 (\mathbb{T}^3 \times (0,T))$. 

In addition, we observe that for any strong limit $u \in L^3 (\mathbb{T}^3 \times (0,T)) \cap L^\infty ((0,T); L^2 (\mathbb{T}^3))$ of any subsequence $\{ u^{\nu_k} \}$ converging strongly in $L^3 (\mathbb{T}^3 \times (0,T))$, it holds that $u^{\nu_k} \rightarrow u$ in $\cap_{\theta > 0} L^3 ((0,T); B^{\alpha-\theta}_{3,\infty} (\mathbb{T}^3)) \cap L^\infty ((0,T); H^{\beta - \theta} (\mathbb{T}^3))$ and the strong convergence holds for any $\theta > 0$, which can be seen by arguing by contradiction. In addition, by using the boundedness of $\{ u^{\nu_k} \}$ in $L^\infty ((0,T); H^{\beta} (\mathbb{T}^3))$, we can extract a subsequence from the earlier subsequence, which converges weak-$*$ to $u \in L^\infty ((0,T); H^{\beta} (\mathbb{T}^3))$.

The strong convergence in $L^3 (\mathbb{T}^3 \times (0,T))$ is sufficient for the limit $u$ to be a weak solution of the Euler equations in the standard sense, such that it satisfies equation \eqref{weakformulationeuler} (cf. \cite{duchon,drivaseyink}). Next, we show that the limit solution $u$ obeys the local helicity balance \eqref{localhelicitybalance}. The sequence of viscous solutions $\{ u^\nu \}$ of the Navier-Stokes equations satisfies the following equation 
\begin{align*}
&\int_0^T \bigg[ - \bigg\langle 2 \omega^\nu u^\nu, \partial_t \varphi \bigg\rangle + \bigg\langle  2 \nabla \cdot (p^\nu \omega^\nu) + \nabla \cdot \big[ 2 (u^\nu \cdot \omega^\nu) u^\nu - \lvert u^\nu \rvert^2 \omega^\nu \big], \varphi \bigg\rangle \bigg] dt  \\
&= - \int_0^T \bigg\langle 4 \nu \nabla u^\nu : \nabla \omega^\nu ,  \varphi \bigg\rangle dt. 
\end{align*}
We observe that the left-hand side is uniformly bounded in $\nu$ in the space $W^{-1,1} ((0,T); W^{-2,1} (\mathbb{T}^3))$. Therefore, we know that the distributional limit $\lim_{\nu \rightarrow 0} \big( 4 \nu \nabla u^\nu : \nabla \omega^\nu \big)$ exists. We will call this limit $D_1 (u,\omega) - \frac{1}{2} D_2 (u, \omega)$. Moreover, due to the strong convergence of (a subsequence of) $\{ u^\nu \}$ in $L^3 ((0,T);B^{\alpha - \theta}_{3,\infty} (\mathbb{T}^3))$, we conclude that (by using Lemma \ref{paraproductlemma})
\begin{align*}
&\int_0^T \bigg[ - \bigg\langle 2 \omega^\nu u^\nu, \partial_t \varphi \bigg\rangle + \bigg\langle  2 \nabla \cdot (p^\nu \omega^\nu) + \nabla \cdot \big[ 2 (u^\nu \cdot \omega^\nu) u^\nu - \lvert u^\nu \rvert^2 \omega^\nu \big], \varphi \bigg\rangle \bigg] dt \\
&\xrightarrow[]{\nu \rightarrow 0} \int_0^T \bigg[ - \bigg\langle 2 \omega u, \partial_t \varphi \bigg\rangle + \bigg\langle  2 \nabla \cdot (p \omega) + \nabla \cdot \big[ 2 (u \cdot \omega) u - \lvert u \rvert^2 \omega \big], \varphi \bigg\rangle \bigg] dt.
\end{align*}
Therefore, we conclude that $u$ satisfies the local helicity balance \eqref{localhelicitybalance}. This conclusion could also have been obtained directly by observing that since $u$ is a weak solution of the Euler equations, by Proposition \ref{funcvorticityprop} it follows that $u$ is a functional vorticity solution of the Euler equations because $u \in L^3 ((0,T); B^{1/2+}_{3,\infty}) \cap L^\infty ((0,T); H^{1/2+} (\mathbb{T}^3))$. 

By Theorem \ref{helicitybalance} we then infer that $u$ obeys the equation of local helicity balance \eqref{localhelicitybalance}. However, the first approach taken is more informative, as we find that the resulting defect terms $D_1 (u,\omega) - \frac{1}{2} D_2 (u, \omega)$ arise as a distributional limit $\lim_{\nu \rightarrow 0} \big( 4 \nu \nabla u^\nu : \nabla \omega^\nu \big)$. If the limit is nonzero (which could potentially happen if $\frac{1}{2} < \alpha \leq \frac{2}{3}$ in the assumptions of the theorem, instead of $\alpha > \frac{2}{3}$), it can be regarded as some sort of `dissipation anomaly'. 

Then by Proposition \ref{conservationbesov} we conclude that the limiting solution $u$ has zero defect terms (i.e. $D_1 (u, \omega) = D_2 (u, \omega) = 0$), as one can choose $\theta$ sufficiently small such that $\alpha - \theta > \frac{2}{3}$. 

Finally, by proceeding like in the proof of Theorem \ref{conservationeuler}, we find that the solution $u$ has constant helicity. Note that we need that $u \in L^\infty ((0,T); H^{1/2} (\mathbb{T}^3))$, as otherwise we cannot make sense of the helicity pointwise. 
\end{proof}

\section{On triple correlation structure functions in helical turbulence} \label{triplecorrsection}
It is well known that in the Kolmogorov theory of turbulence, one can obtain statistical scaling laws for turbulent flows by using structure functions. In particular, the Kolmogorov-$\frac{4}{5}$ law provides an exact expression for the third-order longitudinal structure function in terms of the increment length and the energy dissipation rate \cite{frisch}. Similar results exist in helical turbulence, see for example \cite{alexakis} and references therein. 

Before we explicitly state several scaling laws for structure functions that are known in helical turbulence, we introduce some notation. We will use the subscripts $L$ and $P$ to denote the longitudinal and transversal parts of a difference increment, i.e. for a general vector field $\psi$ we define
\begin{equation}
\delta \psi_L (\xi;x,t) \coloneqq \big( \delta \psi (\xi;x,t) \cdot \widehat{\xi} \; \big) \xi, \quad \delta \psi_{P} (\xi;x,t) \coloneqq \delta \psi(\xi;x,t) - \delta \psi_L (\xi; x,t),
\end{equation}
where $\widehat{\xi}$ denotes the normalised unit vector in the direction of $\xi \in \mathbb{R}^3 \backslash \{ 0 \}$. 

In \cite{gomez} the following scaling law was derived, which will be our main concern in this section
\begin{equation} \label{scalinglaw}
\big\langle \delta u_L (\xi;\cdot,\cdot) \big( \delta u (\xi;\cdot,\cdot) \cdot \delta \omega (\xi;\cdot,\cdot) \big) \rangle - \frac{1}{2} \langle \delta \omega_L (\xi;\cdot,\cdot) \lvert \delta u(\xi;\cdot,\cdot) \rvert^2 \rangle = - \frac{4}{3} \tilde{\epsilon} \xi,
\end{equation}
where $\langle \cdot \rangle$ denotes the ensemble average (to which we will attach rigorous meaning later), while $\tilde{\epsilon} = \nu \big\langle \nabla u : \nabla \omega \rangle$ denotes the average rate of change of the (global) helicity. 

In this section, we combine the methods of \cite{duchon,eyinklocal} (see also \cite{nie,novackscaling}) together with the interpretation of the local helicity balance at low regularity by using paradifferential calculus, which was introduced in Section \ref{eulersection} of this paper. This will allow us to prove a version of equation \eqref{scalinglaw} for functional vorticity solutions.

There is an additional scaling law in helical turbulence that should be mentioned here, which is an analogue of the Kolmogorov $\frac{4}{5}$-law \cite{lvov,kurien,kurienreflection,chkhetiani} (see also \cite{ye})
\begin{equation}
\big\langle \delta u_L (\xi;\cdot,\cdot) \big[ u_P (\cdot+\xi,\cdot) \times u_P (\cdot,t) \big] \big\rangle = \frac{4}{15} \tilde{\epsilon} \lvert \xi \rvert^2,
\end{equation}
where as before, $\langle \cdot \rangle$ denotes the ensemble average (to which we will attach rigorous meaning later).
As was mentioned before, in this work we will focus on proving a version of relation \eqref{scalinglaw} for functional vorticity solutions of the Euler equations. We first introduce the following objects (where $r > 0$, $\varphi  \in \mathcal{D} (\mathbb{T}^3 \times (0,T); \mathbb{R}^3)$ and $d S (\xi)$ denotes the unit Haar measure on the two-dimensional sphere $\mathbb{S}^2$)
\begin{align}
\mathcal{S}_1 (u, \omega, \varphi, r) &\coloneqq \int_{\mathbb{S}^2} d S (\xi) \big\langle \big( \delta u (r \xi;\cdot,t) \cdot \xi \big) \big( \delta \omega (r \xi;\cdot,t) \cdot \delta u (r \xi;\cdot,t) \big), \varphi \big\rangle, \label{sphericalaverage1} \\
\mathcal{S}_2 (u, \omega, \varphi, r) &\coloneqq \int_{\mathbb{S}^2} d S (\xi) \big\langle \big( \delta \omega (r \xi;\cdot,t) \cdot \xi \big) \lvert \delta u (r \xi;\cdot,t) \rvert^2, \varphi \big\rangle. \label{sphericalaverage2}
\end{align} 
We would first like to make a comment on how the objects in equations \eqref{sphericalaverage1} and \eqref{sphericalaverage2} are defined. Similarly to Remark \ref{defectidentityremark}, one can rewrite the duality brackets as follows
\begin{align}
\mathcal{S}_1 (u, \omega, \varphi, r) &= \int_{\mathbb{S}^2} d S (\xi) \bigg\langle \omega, (\delta u (-r \xi ; \cdot , t) \cdot \varphi (\cdot - r \xi, t)) \delta u (- r \xi ; \cdot , t) \nonumber \\
&- (\delta u (r \xi ; \cdot, t) \cdot \varphi (\cdot, t)) \delta u (r \xi ; \cdot, t) \bigg\rangle, \label{dualsphericalaverage1} \\
\mathcal{S}_2 (u, \omega, \varphi, r) &= \int_{\mathbb{S}^2} d S (\xi) \bigg\langle \omega, \lvert \delta u \rvert^2 (-r \xi ; \cdot, t) \varphi (\cdot - r \xi , t) - \lvert \delta u \rvert^2 (r \xi ; \cdot, t) \varphi (\cdot, t) \bigg\rangle. \label{dualsphericalaverage2}
\end{align} 
As was noted in Remark \ref{defectidentityremark}, the integrands are bounded and continuous in $\xi$, and hence are measurable. Therefore the integrals over $\mathbb{S}^2$ are well-defined. 

There is also a different way to see that the spherical averages are well-defined. Let us introduce the following spherical averaging operator
\begin{equation}
A_r f (x,t) \coloneqq \int_{\mathbb{S}^2} f (x - r \xi, t ) d S (\xi), \quad x \in \mathbb{T}^3.
\end{equation}
We observe that for $f \in L^1 (\mathbb{T}^3)$ (where we ignore the time dependency for the moment), the object $A_r f$ is a measurable function which is finite for almost every $x \in \mathbb{T}^3$ and in fact $A_r f \in L^1 (\mathbb{T}^3)$. Moreover, the spherical averaging operator can be viewed as a Fourier multiplier, see \cite[~Section 6.5]{grafakos} for further details. This observation follows from rewriting $A_r f$ as a convolution of $f$ with the indicator function $1_{\mathbb{S}^2}$ on $\mathbb{R}^3$.

In the case of the whole space $\mathbb{R}^3$, the Fourier multiplier (i.e. the Fourier transform of $1_{\mathbb{S}^2}$ on $\mathbb{R}^3$) can be computed explicitly in terms of Bessel functions, see \cite[~Appendix B.4]{grafakos} for details. Then from this explicit representation one can deduce that the Fourier multiplier is $C^\infty$.
The fact that the spherical average is a Fourier multiplier allows us to make the following observation on the Littlewood-Paley blocks (see for example \cite{boutrosnonuniqueness,bahouri} for an introduction to Littlewood-Paley theory)
\begin{equation} \label{convolutionequality}
\Delta_j (A_r f) = A_r (\Delta_j f).
\end{equation}
We note that $A_r f$ can be understood as a distributional convolution, and one can check that the equality holds at the level of the Fourier coefficients.

We note that the operator $A_r$ is a bounded operator from $L^p (\mathbb{T}^3)$ to $L^p (\mathbb{T}^3)$ (for $1 \leq p \leq \infty$) with fixed $r > 0$, by using the Minkowski inequality. This implies the following inequality
\begin{equation} \label{Lpbound}
\lVert A_r \Delta_j f \rVert_{L^p} \leq C_r \lVert \Delta_j f \rVert_{L^p}.
\end{equation}
By using equations \eqref{convolutionequality} and \eqref{Lpbound} we then conclude that $A_r f$ has the same Besov regularity as $f$ (by using the definition of the Besov norms in terms of Littlewood-Paley blocks), so that
\begin{equation*}
\lVert A_r f \rVert_{B^s_{p,q}} \leq C_r \lVert f \rVert_{B^s_{p,q}},
\end{equation*}
for any $s \in \mathbb{R}$ and $1 \leq p, q \leq \infty$. Therefore we conclude that for $u \in L^3 ((0,T); W^{1/2+,3} (\mathbb{T}^3))$, we have the following regularity results
\begin{align}
&\left\lVert \int_{\mathbb{S}^2} d S (\xi) \big( \delta u (r \xi;x,t) \cdot \xi \big) \big( \delta \omega (r \xi;x,t) \cdot \delta u (r \xi;x,t) \big) \right\rVert_{L^1 (B^{-1/2}_{1,3} )} \leq C_r, \label{averageinclusion1} \\
&\bigg\lVert \int_{\mathbb{S}^2} d S (\xi)  \big( \delta \omega (r \xi;x,t) \cdot \xi \big) \lvert \delta u (r \xi;x,t) \rvert^2 \bigg\rVert_{L^1 (B^{-1/2}_{1,3} )} \leq C_r, \label{averageinclusion2}
\end{align}
and we emphasise that these Besov norms depend on $r > 0$. Similarly, we have that
\begin{align}
&\bigg\lVert \int_{\mathbb{S}^2} d S (\xi) \bigg[ (\delta u (-r \xi ; x , t) \cdot \varphi (x - r \xi, t)) \delta u (- r \xi ; x , t) - (\delta u (r \xi ; x, t) \cdot \varphi (x, t)) \delta u (r \xi ; x, t) \bigg] \bigg\rVert_{L^{3/2} (B^{1/2+}_{3/2,3} )} \nonumber \\
&\leq C_r, \label{dualinclusion1} \\
&\bigg\lVert \int_{\mathbb{S}^2} d S (\xi) \bigg[ \lvert \delta u \rvert^2 (-r \xi ; x, t) \varphi (x - r \xi , t) - \lvert \delta u \rvert^2 (r \xi ; x, t) \varphi (x, t) \bigg] \bigg\rVert_{L^{3/2} (B^{1/2+}_{3/2,3} )} \leq C_r. \label{dualinclusion2}
\end{align} 
Therefore, by an adaption of the proof of Proposition \ref{defectidentitylemma}, one can move the integral over $\xi$ inside the duality bracket in equations \eqref{sphericalaverage1}-\eqref{dualsphericalaverage2}. One first introduces a mollification of $\omega$, then applies the Fubini-Tonelli theorem and afterward uses the bounds on the operator $A_r$ that we have stated above. Therefore we will introduce the following notation
\begin{align}
S_1 (u, \omega, r) (x,t) &\coloneqq \int_{\mathbb{S}^2} d S (\xi) \big( \delta u (r \xi;x,t) \cdot \xi \big) \big( \delta \omega (r \xi;x,t) \cdot \delta u (r \xi;x,t) \big), \\
S_2 (u, \omega, r) (x,t) &\coloneqq \int_{\mathbb{S}^2} d S (\xi) \big( \delta \omega (r \xi;x,t) \cdot \xi \big) \lvert \delta u (r \xi;x,t) \rvert^2,
\end{align} 
where by equations \eqref{averageinclusion1} and \eqref{averageinclusion2} it follows that $S_1 (u, \omega, r)$ and $S_2 (u, \omega, r)$ belong to $L^1 ((0,T); B^{-1/2}_{1,3} (\mathbb{T}^3))$ (where the bounds of their norms depend on $r > 0$).

If one wishes to obtain bounds on the Besov norm of the spherical average which are uniform in $r$, then one has to use estimates on the spherical maximal function, which is defined by (due to the periodicity we can restrict the range of $r$)
\begin{equation}
M_S f \coloneqq \sup_{0 < r \leq 1} A_r f. 
\end{equation}
The $L^p (\mathbb{R}^d)$ to $L^p (\mathbb{R}^d)$ boundedness of this maximal function operator was originally proven in \cite{stein}, if $\frac{d}{d-1} < p \leq \infty$ and $d \geq 3$. The boundedness of the spherical maximal operator can be extended to the torus by a straightforward adaption of the proof for the case of the whole space $\mathbb{R}^d$. Therefore for $d = 3$ we have that
\begin{equation} \label{uniformLpbound}
\sup_{r > 0} \lVert A_r \Delta_j f \rVert_{L^p} \lesssim \lVert \Delta_j f \rVert_{L^p}, \quad \frac{3}{2} < p \leq \infty.
\end{equation}
If a function belongs to the space $B^{1/2+}_{3/2,3} (\mathbb{T}^3)$, by a Besov embedding (see for example \cite{bahouri}) there exists a suitably small $\gamma > 0$ such that it belongs to the space $B^{1/2+}_{3/2+\gamma,3} (\mathbb{T}^3)$. Therefore by equations \eqref{convolutionequality} and \eqref{uniformLpbound}, the spherical averages in equations \eqref{dualinclusion1} and \eqref{dualinclusion2} belong to the space $L^{3/2} ((0,T); B^{1/2+}_{3/2+\gamma,3} (\mathbb{T}^3))$ uniformly in $r$ (again for some suitably chosen small $\gamma$ which depends on the regularity exponent).

By a Besov embedding, one obtains that $B^{-1/2}_{1,3} (\mathbb{T}^3) \subset B^{-3/2 - 3 \gamma}_{3/2 ( 1 - 3/2 \gamma)^{-1}, 3} (\mathbb{T}^3)$ for a suitably small $\gamma > 0$. Then again by using \eqref{convolutionequality} and \eqref{uniformLpbound}, one finds that the spherical averages in equations \eqref{averageinclusion1} and \eqref{averageinclusion2} belong to the space $L^1 ( (0,T); B^{-3/2 - 3 \gamma}_{3/2 ( 1 - 3/2 \gamma)^{-1}, 3} (\mathbb{T}^3))$, uniformly in $r$. 

Now that the interpretation of the spherical averages for the functional vorticity solutions of the Euler equations has been clarified, we turn to the proof of the scaling law \eqref{scalinglaw}. In the proof we will adapt an argument from \cite{novackscaling}, which is the choice of a particular type of radial mollifier.
\begin{theorem} \label{scalingthm}
Let $(u,\omega)$ be a functional vorticity solution of the Euler equations such that $u \in L^3 ((0,T);B^{1/2+}_{3,\infty} (\mathbb{T}^3))$. Then the following distributional limit converges in $W^{-1,1} ((0,T); \linebreak W^{-2,1} (\mathbb{T}^3))$
\begin{align}
&\lim_{r \rightarrow 0} \bigg[ \frac{S_1 (u,\omega,r)}{r} - \frac{1}{2} \frac{S_2 (u,\omega,r)}{r} \bigg].
\end{align}
Moreover, the following distributional equation holds in $W^{-1,1} ((0,T); W^{-2,1} (\mathbb{T}^3))$ (for any $\varphi \in \mathcal{D} (\mathbb{T}^3 \times (0,T);\mathbb{R})$)
\begin{equation}
-\frac{4}{3} \cdot \frac{1}{4} \bigg\langle D_{1} (u, \omega) - \frac{1}{2} D_{2} (u, \omega) , \varphi \bigg\rangle_{x,t} = \bigg\langle \lim_{r \rightarrow 0} \bigg[ \frac{1}{4 \pi} \frac{S_1 (u,\omega,r)}{r} - \frac{1}{2} \cdot \frac{1}{4 \pi} \frac{S_2 (u,\omega,r)}{r} \bigg] , \varphi \bigg\rangle_{x,t},
\end{equation}
which is the inviscid version of the scaling law given in equation \eqref{scalinglaw}.
\end{theorem}

\begin{proof}
Let us first define a new (nonincreasing) $C^\infty$ function $\kappa: \mathbb{R} \rightarrow \mathbb{R}$ by
\begin{equation*}
\kappa (s) = \begin{cases}
1 \quad \text{if } s \leq \frac{3}{4}, \\
0 \quad \text{if } s \geq \frac{5}{4}.
\end{cases}
\end{equation*}
We then define a $C^\infty$ function $\kappa_{\mu,\epsilon}: \mathbb{R}^3 \rightarrow \mathbb{R}$ (for $\mu, \epsilon > 0$)
\begin{equation*}
\widetilde{\kappa}_{\mu,\epsilon} (x) \coloneqq \frac{1}{\epsilon^3 } \kappa \left( 1 + \mu^{-1} \left(\frac{\lvert x \rvert}{\epsilon} -1 \right) \right).
\end{equation*}
Then we define the following smooth mollifier $\kappa_{\mu,\epsilon}$ as follows
\begin{equation*}
\kappa_{\mu,\epsilon} (x) \coloneqq \frac{1}{\lVert \widetilde{\kappa}_{\mu,1} \rVert_{L^1}} \widetilde{\kappa}_{\mu,\epsilon} (x).
\end{equation*}
Similarly to before, we will use the notation $u^{\mu,\epsilon} \coloneqq u * \kappa_{\mu,\epsilon}$. Next we observe that the derivation of equation \eqref{mollifiedbalance} (in the proof of the equation of local helicity balance in Theorem \ref{helicitybalance}) does not rely on any particular choice of mollifier, therefore equation \eqref{mollifiedbalance} also holds with $\phi_\epsilon$ replaced by $\kappa_{\mu,\epsilon}$, which leads to (for any $\varphi \in \mathcal{D} (\mathbb{T}^3 \times (0,T);\mathbb{R})$)
\begin{align}
&\int_0^T \bigg\langle \partial_t \big( \omega^{\mu,\epsilon} u + \omega u^{\mu,\epsilon} \big) + \nabla \cdot (p \omega^{\mu,\epsilon} + p^{\mu,\epsilon} \omega) + \nabla \cdot \big[ (u \cdot \omega^{\mu,\epsilon}) u + (\omega \cdot u^{\mu,\epsilon}) u - (u \cdot u^{\mu,\epsilon}) \omega \big], \varphi \bigg\rangle dt \nonumber \\
&+ \frac{1}{2} \int_0^T \bigg\langle \nabla \cdot \big[ 2((\omega \cdot u) u)^{\mu,\epsilon} - 2(\omega \cdot u)^{\mu,\epsilon} u - ( \lvert u \rvert^2 \omega)^{\mu,\epsilon} + \omega (\lvert u \rvert^2)^{\mu,\epsilon} \big], \varphi \bigg\rangle dt \nonumber \\
&= \int_0^T \bigg\langle - D_{1,\mu,\epsilon} (u, \omega) + \frac{1}{2} D_{2,\mu,\epsilon} (u, \omega) ,  \varphi \bigg\rangle dt, \label{mollifiedsphericalbalance}
\end{align}
where the defect terms are now given by
\begin{align*}
\bigg\langle D_{1,\mu,\epsilon} (u, \omega) - \frac{1}{2} D_{2,\mu,\epsilon} (u, \omega), \varphi \bigg\rangle &= \int_{\mathbb{R}^3} \nabla \kappa_{\mu,\epsilon} (\xi) \cdot \big\langle \delta u (\xi ; \cdot, t) (\delta \omega (\xi ; \cdot, t) \cdot \delta u (\xi ; \cdot, t) ), \varphi \big\rangle d \xi \\
&- \frac{1}{2} \int_{\mathbb{R}^3} \nabla \kappa_{\mu,\epsilon} (\xi) \cdot \big\langle \delta \omega (\xi ; \cdot, t) \lvert \delta u (\xi ; \cdot, t) \rvert^2, \varphi \big\rangle d \xi.
\end{align*}
Note that the defect terms $D_{1,\mu,\epsilon} (u, \omega)$ and $D_{2,\mu,\epsilon} (u, \omega)$ are elements of $L^1 ((0,T);B^{-1/2}_{1,\infty} (\mathbb{T}^3))$. Now because the mollifier $\kappa_{\mu,\epsilon}$ is radial, we will change to spherical coordinates and write (where we make the change of variable $r \xi' \coloneqq \frac{\xi}{\epsilon}$ where we choose $\lvert \xi' \rvert = 1$ and we will relabel $\xi'$ by $\xi$)
\begin{align*}
&\bigg\langle D_{1,\mu,\epsilon} (u, \omega) - \frac{1}{2} D_{2,\mu,\epsilon} (u, \omega), \varphi \bigg\rangle \\
&= \int_{0}^\infty dr \frac{1}{\lVert \widetilde{\kappa}_{\mu,1} \rVert_{L^1}} \kappa' ( 1 + \mu^{-1} (r -1) ) \frac{r^3}{\epsilon \mu r} \bigg[ \int_{\mathbb{S}^2} \big\langle \xi \cdot \delta u (\epsilon r \xi ; \cdot,t) \big(\delta \omega (\epsilon r \xi ;\cdot,t) \cdot \delta u (\epsilon r \xi ; \cdot, t) \big), \varphi \big\rangle d S (\xi) \\
&- \frac{1}{2} \int_{\mathbb{S}^2} \big\langle \xi \cdot \delta \omega (\epsilon r \xi ; \cdot,t) \lvert \delta u (\epsilon r \xi ;\cdot,t) \rvert^2, \varphi \big\rangle d S (\xi) \bigg] \\
&= \int_0^\infty dr \frac{1}{\lVert \widetilde{\kappa}_{\mu,1} \rVert_{L^1}} \kappa' ( 1 + \mu^{-1} (r -1) ) \frac{r^3}{\epsilon \mu} \bigg[ \frac{\langle S_1 (u,\omega, \epsilon r), \varphi \rangle - \frac{1}{2} \langle S_2 (u,\omega, \epsilon r), \varphi \rangle}{r} \bigg].
\end{align*}
We now show that the right-hand side converges to $- \frac{3}{4 \pi \epsilon} \langle S_1 (u,\omega, \epsilon), \varphi \rangle + \frac{3}{8 \pi \epsilon} \langle S_2 (u,\omega, \epsilon), \varphi \rangle$, subtracting these terms from the equation gives
\begin{align*}
&\bigg\lvert \int_0^\infty dr \frac{1}{\lVert \widetilde{\kappa}_{\mu,1} \rVert_{L^1}} \kappa' ( 1 + \mu^{-1} (r -1) ) \frac{r^3}{\epsilon \mu} \bigg[ \frac{\langle S_1 (u,\omega, \epsilon r), \varphi \rangle - \frac{1}{2} \langle S_2 (u,\omega, \epsilon r), \varphi \rangle}{r} - \langle S_1 (u,\omega, \epsilon ), \varphi \rangle \\
&+ \frac{1}{2} \langle S_2 (u,\omega, \epsilon ), \varphi \rangle \bigg] \bigg\rvert \\
&\lesssim \frac{1}{\epsilon \mu} \int_{1 - \frac{1}{4} \mu}^{1 + \frac{1}{4} \mu} r^3 \bigg\lvert \frac{\langle S_1 (u,\omega, \epsilon r), \varphi \rangle}{r} - \langle S_1 (u,\omega, \epsilon), \varphi \rangle - \frac{1}{2} \frac{\langle S_2 (u,\omega, \epsilon r), \varphi \rangle}{r} \\
&+ \frac{1}{2} \langle S_2 (u,\omega, \epsilon), \varphi \rangle \bigg\rvert dr \xrightarrow[]{\mu \rightarrow 0} 0,
\end{align*}
where we have used the continuity of the translation operator in Besov spaces. We observe that the left-hand side of equation \eqref{mollifiedsphericalbalance} converges in the sense of distributions as $\mu \rightarrow 0$. This is because $\lim_{\mu \rightarrow 0} \kappa_{\mu,\epsilon} = \epsilon^{-3} 1_{B_\epsilon (0)} \eqqcolon \kappa_{0,\epsilon}$ in $L^p (\mathbb{R}^3)$ for any $1 \leq p < \infty$ (for fixed $\epsilon > 0$), where $B_\epsilon (0)$ is the ball of radius $\epsilon$ around $0$. Therefore we find that
\begin{align}
&\int_0^T \bigg\langle \partial_t \big( \omega^{0,\epsilon} u + \omega u^{0,\epsilon} \big) + \nabla \cdot (p \omega^{0,\epsilon} + p^{0,\epsilon} \omega) + \nabla \cdot \big[ (u \cdot \omega^{0,\epsilon}) u + (\omega \cdot u^{0,\epsilon}) u - (u \cdot u^{0,\epsilon}) \omega \big], \varphi \bigg\rangle dt \nonumber \\
&+ \frac{1}{2} \int_0^T \bigg\langle \nabla \cdot \big[ 2((\omega \cdot u) u)^{0,\epsilon} - 2(\omega \cdot u)^{0,\epsilon} u - ( \lvert u \rvert^2 \omega)^{0,\epsilon} + \omega (\lvert u \rvert^2)^{0,\epsilon} \big], \varphi \bigg\rangle dt \nonumber \\
&= \frac{3}{4 \pi} \bigg\langle \frac{S_1 (u,\omega,\epsilon)}{\epsilon} - \frac{1}{2} \frac{S_2 (u,\omega,\epsilon)}{\epsilon} , \varphi \bigg\rangle_{x,t}. \label{sphericalaverageexpression}
\end{align}
We observe that as $\epsilon \rightarrow 0$ the left-hand side of equation \eqref{sphericalaverageexpression} converges in $W^{-1,1} ((0,T); W^{-2,1} (\mathbb{T}^3))$ to the left-hand side of equation \eqref{limithelicitybalance}. Therefore we can conclude that the limit of the spherical averages $\lim_{\epsilon \rightarrow 0} \bigg[ \frac{1}{\epsilon} S_1 (u,\omega,\epsilon) - \frac{1}{2 \epsilon} S_2 (u,\omega,\epsilon) \bigg]$ is well-defined as a distribution and is equal to the right-hand side of equation \eqref{limithelicitybalance}, which leads to the following equality
\begin{equation*}
\bigg\langle - D_{1} (u, \omega) + \frac{1}{2} D_{2} (u, \omega) ,  \varphi \bigg\rangle_{x,t} = \frac{3}{4 \pi} \bigg\langle \lim_{\epsilon \rightarrow 0} \bigg[ \frac{S_1 (u,\omega,\epsilon)}{\epsilon} - \frac{1}{2} \frac{S_2 (u,\omega,\epsilon)}{\epsilon} \bigg] , \varphi \bigg\rangle_{x,t}.
\end{equation*}
We observe that the left-hand side can be formally identified (up to a minus sign) with the average rate of change of the helicity in the vanishing viscosity limit if one divides by a factor of $4$. Moreover, we note that the factors $(4 \pi)^{-1}$ are part of the spherical averages. Therefore we can rewrite this relation as follows
\begin{equation}
-\frac{4}{3} \cdot \frac{1}{4} \bigg\langle D_1 (u, \omega) - \frac{1}{2} D_2 (u, \omega), \varphi \bigg\rangle_{x,t} = \bigg\langle \lim_{\epsilon \rightarrow 0} \bigg[ \frac{1}{4 \pi} \frac{S_1 (u,\omega,\epsilon)}{\epsilon} - \frac{1}{2} \cdot \frac{1}{4 \pi} \frac{S_2 (u,\omega,\epsilon)}{\epsilon} \bigg] , \varphi \bigg\rangle_{x,t},
\end{equation}
which holds in $W^{-1,1} ((0,T); W^{-2,1} (\mathbb{T}^3))$ and is in agreement with equation \eqref{scalinglaw}, and therefore concludes the proof.
\end{proof}

Therefore we can conclude that the scaling law derived in \cite{gomez} holds for functional vorticity solutions of the Euler equations with sufficient regularity. 
\begin{remark}
For the sake of completeness, we will now explain how we have interpreted the ensemble averages in equation \eqref{scalinglaw} in the proof of Theorem \ref{scalingthm}. For a function $f : \mathbb{T}^3 \times \mathbb{T}^3 \times [0,T] \rightarrow \mathbb{R}$ (we will denote the arguments by $z, x$ and $t$, respectively) we define (cf. \cite{duchon,eyinklocal})
\begin{equation} \label{ensembleaveragedefinition}
\langle f \rangle (r,x,t) \coloneqq \int_{\mathbb{S}^2} d S (\widehat{z}) f (r \widehat{z}, x,t),
\end{equation}
where we have written $z = r \widehat{z}$ such that $\lvert \widehat{z} \rvert = 1$ and $r \geq 0$. As has been mentioned before in this section, such objects can be defined even if $f$ has negative Besov regularity. Now we observe that Theorem \ref{scalingthm} establishes equation \eqref{scalinglaw} in the limit $\xi \rightarrow 0$. This means that we need to show that the following limit converges (in the sense of distributions)
\begin{equation*}
\lim_{r \rightarrow 0} \frac{\langle f \rangle (r,x,t)}{r}.
\end{equation*}
For smooth functions, such a limit can only converge if $f(0,x,t) \equiv 0$. In the case of Theorem \ref{scalingthm}, because $f$ consists of increments of the velocity and vorticity this condition is automatically satisfied. Then by using the (mollified) local helicity balance we show that this limit converges in the sense of distributions and show that it is related to the defect terms, which then concludes the proof of the scaling law \eqref{scalinglaw}.
\end{remark}

\section{On the conservation of magnetic helicity in the MHD equations} \label{MHDsection}
Now we turn our attention to the inviscid and irresistive MHD equations. In this section we will provide sufficient conditions for the conservation of magnetic helicity, which are different from those stated in \cite{kang}. We note that in simply connected domains the magnetic helicity as defined in this paper is gauge-independent. We refer to \cite{faraco} (and references therein) for gauge-invariant definitions of the magnetic helicity in general domains. 

As was already noted in the introduction, for weak solutions it need not be true that $\nabla \cdot B$ remains zero, if the magnetic field is initially divergence-free. This point will be addressed further in Section \ref{divergencesection}. In this section, however, we will assume throughout that the magnetic field is weakly divergence-free.

We now define a weak solution to the MHD equations \eqref{MHD1}-\eqref{MHD3}.
\begin{definition} \label{weaksolMHDdivergencefree}
We call a pair $(u,B) \in L^\infty ((0,T); L^2 (\mathbb{T}^3))$ a weak solution of the MHD equations \eqref{MHD1}-\eqref{MHD3} if for all divergence-free $\psi_1, \psi_2 \in \mathcal{D} (\mathbb{T}^3 \times (0,T);\mathbb{R}^3)$ it holds that
\begin{align}
&\int_0^T \int_{\mathbb{T}^3} \bigg[ u \partial_t \psi_1 + u \otimes u : \nabla \psi_1 - B \otimes B : \nabla \psi_1 \bigg] dx dt = 0, \label{mhdweakform1} \\
&\int_0^T \int_{\mathbb{T}^3} \bigg[ B \partial_t \psi_2 + u \otimes B : \nabla \psi_2 - B \otimes u : \nabla \psi_2 \bigg] dx dt = 0, \label{mhdweakform2}
\end{align}
and $B$ and $u$ is weakly divergence-free. 
\end{definition}
We now prove the following sufficient condition for conservation of the magnetic helicity.
\begin{theorem} \label{maghelicityconservation}
Let $(u,B)$ be a weak solution (in the sense of Definition \ref{weaksolMHDdivergencefree}) of the MHD equations \eqref{MHD1}-\eqref{MHD3}. Moreover, let $0 < \alpha_1 < \alpha_2 < 1$  such that $u \in L^2 ((0,T); B^{\alpha_2}_{4,\infty} (\mathbb{T}^3) )$ and $B \in L^2 ((0,T); B^{-\alpha_1}_{4,\infty} (\mathbb{T}^3))$. Then the magnetic helicity (as given in equation \eqref{magnetichelicity}) is conserved.
\end{theorem}
The proof will use a commutator estimate, of the type first introduced in \cite{constantin} (see also \cite{titi2018}).
\begin{proof}
We first recall the following commutator identity from \cite{kang} (see also \cite{titi2018})
\begin{equation} \label{curlcommutator}
(u \times B)^\epsilon = u^\epsilon \times B^\epsilon + r_\epsilon (u,B) - (u - u^\epsilon) \times (B - B^\epsilon),
\end{equation}
where we have defined
\begin{equation} \label{repsilondef}
r_\epsilon (u,B) = \int_{\mathbb{R}^3} \phi_\epsilon (\xi) \big( \delta u (\xi ; x,t) \times \delta B (\xi ; x,t) \big) d \xi. 
\end{equation}
We observe that identity \eqref{curlcommutator} holds in the space $L^1 (\mathbb{T}^3 \times (0,T))$. Then by proceeding as in \cite{kang}, we find that the difference in magnetic helicity in time can be bounded as follows (for $0 < t_1 < t_2 < T$)
\begin{align*}
&\bigg\lvert \int_{\mathbb{T}^3} A^\epsilon \cdot B^\epsilon (x,t_2) dx - \int_{\mathbb{T}^3} A^\epsilon \cdot B^\epsilon (x,t_1) dx \bigg\rvert = \left\lvert \int_{t_1}^{t_2} \int_{\mathbb{T}^3} \bigg[ A^\epsilon \cdot (\nabla \times (u \times B)^\epsilon) + B^\epsilon \cdot (u \times B)^\epsilon \bigg] dx dt \right\rvert \\
&= 2 \bigg\lvert \int_{t_1}^{t_2} \int_{\mathbb{T}^3}  B^\epsilon \cdot (u \times B)^\epsilon dx dt \bigg\rvert = 2 \bigg\lvert \int_{t_1}^{t_2} \int_{\mathbb{T}^3} B^\epsilon \cdot \big[ u^\epsilon \times B^\epsilon + r_\epsilon (u,B) - (u - u^\epsilon) \times (B - B^\epsilon) \big] dx dt \bigg\rvert \\
&\leq 2 \bigg\lvert \int_{t_1}^{t_2} \int_{\mathbb{T}^3} B^\epsilon \cdot  r_\epsilon (u,B)  dx dt \bigg\rvert + 2 \bigg\lvert \int_{t_1}^{t_2} \int_{\mathbb{T}^3} B^\epsilon \cdot \big( (u - u^\epsilon) \times (B - B^\epsilon) \big) dx dt \bigg\rvert \eqqcolon I_1 + I_2,
\end{align*}
where in the first line we have used the regularisation of the MHD equations as given in equations \eqref{MHD1}-\eqref{MHD3} and \eqref{mhdpotential1}-\eqref{mhdpotential3} respectively (as well as the divergence-free property of the magnetic field), and in the second line we have used identity \eqref{curlcommutator}. Next we estimate the two integral terms. We find
\begin{align*}
I_2 &\leq \int_{t_1}^{t_2} \lVert B^\epsilon \rVert_{L^4} \lVert u - u^\epsilon \rVert_{L^4} \lVert B - B^\epsilon \rVert_{L^2} dt \leq \epsilon^{-\alpha_1} \lVert B \rVert_{L^2 (B^{-\alpha_1}_{4,\infty})} \cdot \epsilon^{\alpha_2} \lVert u \rVert_{L^2 (B^{\alpha_2}_{4,\infty})} \lVert B \rVert_{L^\infty (L^2)} \\
&\leq \epsilon^{\alpha_2 - \alpha_1} \lVert B \rVert_{L^2 (B^{-\alpha_1}_{4,\infty})} \lVert u \rVert_{L^2 (B^{\alpha_2}_{4,\infty})} \lVert B \rVert_{L^\infty (L^2)}, 
\end{align*}
where we have used inequality \eqref{mollificationbesovineq3}. The estimate of the term $I_1$ proceeds in a similar fashion. Therefore we can infer that
\begin{align*}
&\bigg\lvert \int_{\mathbb{T}^3} A^\epsilon \cdot B^\epsilon (x,t_2) dx - \int_{\mathbb{T}^3} A^\epsilon \cdot B^\epsilon (x,t_1) dx \bigg\rvert \leq C \epsilon^{\alpha_2 - \alpha_1} \lVert B \rVert_{L^2 (B^{-\alpha_1}_{4,\infty})} \lVert u \rVert_{L^2 (B^{\alpha_2}_{4,\infty})} \lVert B \rVert_{L^\infty (L^2)}.
\end{align*}
Therefore, as $\epsilon \rightarrow 0$, the above together with the assumptions of the theorem imply that
\begin{equation}
\int_{\mathbb{T}^3} A \cdot B (x,t_1) dx = \int_{\mathbb{T}^3} A \cdot B (x,t_2) dx,
\end{equation}
for almost every time $t_1, t_2 \in (0,T)$.
Hence we conclude that the magnetic helicity is conserved for almost all times in $(0,T)$. 

\end{proof}

\begin{remark}
We observe that the sufficient criterion for conservation of magnetic helicity that we prove in Theorem \ref{maghelicityconservation}, differs from the criterion proven in \cite{kang}. Indeed, without further spatial $L^p$ regularity assumptions on $B$ (for $p > 2$), one would have to require that $u \in L^1 ((0,T); L^\infty (\mathbb{T}^3))$ to bound integrals of the form $\int_{\mathbb{T}^3} B^\epsilon \cdot \big( (u - u^\epsilon) \times (B - B^\epsilon) \big) dx$. These integrals appear as part of the commutator estimates. However, we observe that the Besov space $B^{\alpha}_{4,\infty} (\mathbb{T}^3)$ embeds into $L^\infty (\mathbb{T}^3)$ only for $\alpha > \frac{3}{4}$. Hence our results are new in the range $0 < \alpha \leq \frac{3}{4}$.  \\
In addition, one can easily see that the proof of Theorem \ref{maghelicityconservation} can be adapted to show that conservation of magnetic helicity still holds if $u \in L^2 ((0,T); B^{\alpha_2}_{p,\infty} (\mathbb{T}^3))$ and $B \in L^2 ((0,T); \linebreak B^{-\alpha_1}_{q,\infty} (\mathbb{T}^3))$ for $\frac{1}{p} + \frac{1}{q} = \frac{1}{2}$ and $0 < \alpha_1 < \alpha_2 < 1$. 
\end{remark}

\section{Remarks on the kinematic dynamo model} \label{dynamosection}
If the effects of the magnetic field on the fluid can be neglected (i.e., the Lorentz force is small), one can take the velocity field to be given and consider the kinematic dynamo model (i.e., restrict to the induction equation in the MHD equations \eqref{MHD1}-\eqref{MHD3}) \cite{davidson}
\begin{equation} \label{inductioneq}
\partial_t B + \nabla \times (B \times u) = 0, \quad \nabla \cdot B = 0,
\end{equation}
where the divergence-free velocity field $u$ is given.
In the viscous case this equation was studied in \cite{friedlander}. We note that solutions of this equation still conserve magnetic helicity. However, unlike the full MHD system, equation \eqref{inductioneq} is linear (in the magnetic field). This makes it possible to introduce an analogue of the functional vorticity solutions to the kinematic dynamo equation, which we call functional weak solutions.
\begin{definition} \label{functionalsolutiondef}
Let $\alpha > 0$ and let $u \in L^\infty ((0,T);H^{\alpha+} (\mathbb{T}^3))$ be given. We call $B \in L^\infty ((0,T);H^{-\alpha} (\mathbb{T}^3))$ a functional weak solution of the kinematic dynamo equation if for all $\psi \in \mathcal{D} (\mathbb{T}^3 \times (0,T);\mathbb{R}^3)$ it satisfies
\begin{equation}
\int_0^T \langle B , \partial_t \psi \rangle dt - \int_0^T \langle B \times u, \nabla \times \psi \rangle dt = 0,
\end{equation}
where $\langle \cdot, \cdot \rangle$ denote the spatial distributional duality brackets as before. Moreover, the magnetic field $B$ is divergence-free in the sense of distributions, i.e. for all $\varphi \in \mathcal{D} (\mathbb{T}^3 \times (0,T); \mathbb{R})$ we have
\begin{equation}
\int_0^T \langle B, \nabla \varphi \rangle dt = 0.
\end{equation}
\end{definition}
\begin{remark}
We observe that it is also possible to define functional weak solutions in terms of the weak formulation of the induction equation in potential form (i.e. equation \eqref{mhdpotential1}), which is
\begin{equation}
\int_0^T \int_{\mathbb{T}^3} A \partial_t \psi dx dt - \int_0^T \langle B \times u, \psi \rangle dt + \int_0^T \langle \tilde{\phi}, \nabla \cdot \psi \rangle dt = 0. 
\end{equation}
Moreover, we note that the $L^\infty$ integrability in time in Definition \ref{functionalsolutiondef} is not strictly necessary, as $L^2$ temporal integrability for $u$ and $B$ is sufficient to make sense of the nonlinear terms. However, to study sufficient conditions for conservation of the magnetic helicity, $L^\infty$ regularity in time is essential. 
\end{remark}
We now prove a sufficient condition for functional weak solutions to conserve magnetic helicity.
\begin{theorem}
Let $B$ be a functional weak solution of the kinematic dynamo equation \eqref{inductioneq} for a given velocity field $u \in L^3 ((0,T);B^{\beta}_{3,\infty} (\mathbb{T}^3))$ such that $B \in L^3 ((0,T);B^{-\alpha}_{3,\infty} (\mathbb{T}^3))$ with $0 < 3\alpha < \beta < 1$. Moreover, let $A$ be the corresponding magnetic potential satisfying equation \eqref{potentialequation}. Then the functional weak solution conserves magnetic helicity, that is
\begin{equation}
\langle B, A \rangle (t_1) = \langle B, A \rangle (t_2),
\end{equation}
for almost all times $t_1, t_2 \in (0,T)$. 
\end{theorem}
\begin{proof}
By proceeding in a similar fashion as in the proof of Theorem \ref{maghelicityconservation}, we find (for almost every $t_1, t_2 \in (0,T)$ with $t_1 < t_2$)
\begin{align*}
&\bigg\lvert \int_{\mathbb{T}^3} A^\epsilon \cdot B^\epsilon (x,t_2) dx - \int_{\mathbb{T}^3} A^\epsilon \cdot B^\epsilon (x,t_1) dx \bigg\rvert \\
&\leq 2 \bigg\lvert \int_{t_1}^{t_2} \big\langle  r_\epsilon (u,B) , B^\epsilon  \rangle dt \bigg\rvert + 2 \bigg\lvert \int_{t_1}^{t_2}\big\langle (u - u^\epsilon) \times (B - B^\epsilon) , B^\epsilon \big\rangle dt \bigg\rvert \eqqcolon I_1 + I_2,
\end{align*}
where $r_\epsilon (u,B)$ was defined in equation \eqref{repsilondef}. By an argument similar to the proof of Proposition \ref{defectidentitylemma}, we obtain that
\begin{equation*}
\int_{t_1}^{t_2} \big\langle  r_\epsilon (u,B) , B^\epsilon  \rangle dt = \int_{t_1}^{t_2} \int_{\mathbb{R}^3} \phi_\epsilon (\xi) \big\langle  \big( \delta u (\xi ; x,t) \times \delta B (\xi ; x,t) \big) , B^\epsilon  \rangle d \xi dt.
\end{equation*}
Using this identity, we can estimate $I_1$ in the following manner (where we fix $\theta > 0$ sufficiently small such that $\beta - 3 \alpha - \theta > 0$)
\begin{align*}
I_1 &\lesssim \epsilon^{-2\alpha} \int_{t_1}^{t_2} \bigg[ \int_{\mathbb{R}^3} \phi_\epsilon (\xi) \lVert \delta u (\xi; \cdot, t) \times \delta B (\xi;\cdot,t) \rVert_{B^{-\alpha}_{3/2,\infty}} d \xi \bigg] \lVert B \rVert_{B^{-\alpha}_{3,\infty}} dt \\
&\lesssim \epsilon^{-2\alpha} \int_{t_1}^{t_2} \bigg[ \int_{\mathbb{R}^3} \phi_\epsilon (\xi) \lVert \delta u (\xi; \cdot, t) \rVert_{B^{\alpha + \theta}_{3,\infty}} d \xi \bigg] \lVert B \rVert_{B^{-\alpha}_{3,\infty}}^2 dt \\
&\lesssim \epsilon^{\beta-3\alpha-\theta} \int_{t_1}^{t_2} \lVert u  \rVert_{B^{\beta}_{3,\infty}} \lVert B \rVert_{B^{-\alpha}_{3,\infty}}^2 dt \xrightarrow[]{\epsilon \rightarrow 0} 0,
\end{align*}
where we have used Lemmas \ref{paraproductlemma}, \ref{diffquotientlemma} and \ref{mollificationbesovlemma}. Then the integral $I_2$ can be estimated as follows 
\begin{align*}
I_2 &\leq 2 \int_{t_1}^{t_2} \lVert (u - u^\epsilon) \times (B - B^\epsilon) \rVert_{B^{-\alpha}_{3/2,\infty}} \lVert B^\epsilon \rVert_{B^\alpha_{3,1}} dt \\
&\lesssim \epsilon^{-2\alpha} \int_{t_1}^{t_2} \lVert u - u^\epsilon \rVert_{B^{\alpha+\theta}_{3,\infty}} \lVert B - B^\epsilon \rVert_{B^{-\alpha}_{3,\infty}} \lVert B \rVert_{B^{-\alpha}_{3,\infty}} dt \\
&\lesssim \epsilon^{\beta-3\alpha-\theta} \int_{t_1}^{t_2} \lVert u  \rVert_{B^{\beta}_{3,\infty}} \lVert B \rVert_{B^{-\alpha}_{3,\infty}}^2 dt \xrightarrow[]{\epsilon \rightarrow 0} 0. 
\end{align*}
In the above we have used the estimates from Lemmas \ref{paraproductlemma} and \ref{mollificationbesovlemma}. As both $I_1$ and $I_2$ go to zero as $\epsilon \rightarrow 0$, we conclude therefore that the functional weak solution conserves magnetic helicity. 
\end{proof}

\section{On the divergence-free property of the magnetic field} \label{divergencesection}
For classical solutions of the MHD equations, by taking the divergence of equation \eqref{MHD2}, one can observe that the divergence of the magnetic field satisfies the following equation
\begin{equation} \label{divtransportequation}
\partial_t \Div B + u \cdot \nabla \Div B = 0.
\end{equation}
Therefore, if classical solutions of the MHD equations have a magnetic field which is divergence-free at the initial time, the magnetic field will remain divergence-free. However, if we consider weak solutions and hence the vector field in the transport equation \eqref{divtransportequation} has low regularity, the uniqueness of weak solutions to equation \eqref{divtransportequation} is not guaranteed. Examples of such nonuniqueness results can be found in \cite{modena,modena2019}.

Therefore it is not guaranteed that $\Div B$ will remain zero in the context of weak solutions if it is initially so. To the best knowledge of the authors, this point has been overlooked in previous work concerning low-regularity weak solutions of the MHD equations. Moreover, the derivation of the conservation laws of the MHD equations relies on the assumption of the magnetic field being divergence-free. Therefore, any sufficient condition for these conservation laws to hold must be accompanied with a (weakly) divergence-free assumption on the magnetic field. 

As has been mentioned before, the divergence-free condition for the magnetic field implies the absence of magnetic monopoles in the solution and therefore is physically significant. In Corollary \ref{viscositylimitcorollary} we show that weak solutions of the ideal MHD equations which are weak-$*$ limits of Leray-Hopf solutions of the viscous MHD system satisfy the divergence-free condition for all time. This result could therefore be interpreted as a (partial) selection criterion for physically relevant weak solutions of the ideal MHD system.

We first introduce a new weak formulation of the MHD equations in which potentially $\Div B \neq 0$.
\begin{definition} \label{newweaksolMHDdef}
We call a pair $(u,B)$ a weak solution of the ideal MHD equations if $u, B, \Div B \in L^\infty ((0,T); L^2 (\mathbb{T}^3))$ and for all divergence-free $\psi_1, \psi_2 \in \mathcal{D} (\mathbb{T}^3 \times (0,T);\mathbb{R}^3)$ and $\varphi \in \mathcal{D} (\mathbb{T}^3 \times (0,T);\mathbb{R})$ the following equations hold
\begin{align}
&\int_0^T \int_{\mathbb{T}^3} \bigg[ u \partial_t \psi_1 + u \otimes u : \nabla \psi_1 - B \otimes B : \nabla \psi_1 - (\Div B ) B \psi_1 \bigg] dx dt = 0, \label{newweakformMHD1} \\
&\int_0^T \int_{\mathbb{T}^3} \bigg[ B \partial_t \psi_2 + u \otimes B : \nabla \psi_2 - B \otimes u : \nabla \psi_2 - u (\Div B) \psi_2 \bigg] dx dt = 0, \label{newweakformMHD2} \\
&\int_0^T \int_{\mathbb{T}^3} u \cdot \nabla \varphi dx dt = 0. \label{newweakformMHD3}
\end{align}
\end{definition}
\begin{remark} \label{divtransportremark}
For weak solutions of the ideal MHD equations in the sense of Definition \ref{newweaksolMHDdef}, one can derive the evolution equation for $\Div B$ by taking $\psi_2 = \nabla \Phi$ in equation \eqref{newweakformMHD2}, where $\Phi \in \mathcal{D} (\mathbb{T}^3 \times (0,T);\mathbb{R})$. This yields
\begin{equation} \label{divBweakformulation}
\int_0^T \int_{\mathbb{T}^3} \bigg[ (\Div B) \partial_t \Phi + u (\Div B) \cdot \nabla \Phi \bigg] dx dt = 0.
\end{equation}
We note that equation \eqref{divBweakformulation} is the weak formulation of equation \eqref{divtransportequation}.
\end{remark}
\begin{remark}
We also observe that different regularity assumptions in Definition \ref{newweaksolMHDdef} would also ensure that the terms $u \Div B$, $B \Div B$ in equations \eqref{newweakformMHD1} and \eqref{newweakformMHD2} are well-defined. In particular, let $u \in L^{p_1} ((0,T); B^{\alpha_1}_{p_2,\infty} (\mathbb{T}^3))$ and $B \in L^{q_1} ((0,T); B^{\alpha_2}_{q_2,\infty} (\mathbb{T}^3))$ with
\begin{equation*}
\alpha_2 > \frac{1}{2}, \quad q_1, q_2 \geq 2, \quad \alpha_1 + \alpha_2 > 1, \quad \frac{1}{p_1} + \frac{1}{q_1} = \frac{1}{p_2} + \frac{1}{q_2} = 1.
\end{equation*}
Under these regularity assumptions (as well as $u, B \in L^\infty ((0,T); L^2 (\mathbb{T}^3))$), equations \eqref{MHD1} and \eqref{MHD2} have the following weak formulation
\begin{align}
&\int_0^T \int_{\mathbb{T}^3} \bigg[ u \partial_t \psi_1 + u \otimes u : \nabla \psi_1 - B \otimes B : \nabla \psi_1 \bigg] dx dt - \int_0^T \langle (\Div B ) B , \psi_1 \rangle dt = 0, \\
&\int_0^T \int_{\mathbb{T}^3} \bigg[ B \partial_t \psi_2 + u \otimes B : \nabla \psi_2 - B \otimes u : \nabla \psi_2  \bigg] dx dt - \int_0^T \langle u \Div B , \psi_2 \rangle dt = 0,
\end{align}
where we recall that $\langle \cdot, \cdot \rangle$ are the spatial distributional duality brackets. Note that it follows from Lemma \ref{paraproductlemma} that in this case $(\Div B) B$ and $u \Div B$ are well-defined as distributions.
\end{remark}
As was noted before, it is conceivable that there exist weak solutions of the ideal MHD equations such that the magnetic field is divergence-free initially, but it does not remain divergence-free. This is because solutions to the transport equation \eqref{divtransportequation} need not to be unique for $u \in L^\infty ((0,T); L^2 (\mathbb{T}^3))$ (see for example \cite{huysmans}). Consequently, one can make the following elementary observation.
\begin{proposition}
Let $(u,B)$ be a weak solution of the ideal MHD equations in the sense of Definition \ref{newweaksolMHDdef}, such that $\Div B \lvert_{t = 0} = 0$ and $u \in L^1 ((0,T); W^{1,1} (\mathbb{T}^3))$. Then $\Div B = 0$ for all time $t > 0$.
\end{proposition}
\begin{proof}
As the transporting vector field $u$ in equation \eqref{divtransportequation} falls within the DiPerna-Lions class, therefore there exists a unique renormalised solution to equation \eqref{divtransportequation} \cite{diperna}. The unique solution therefore has to coincide with $\Div B$ and as $\Div B \equiv 0$ is a solution of the equation, it follows that if the magnetic field is divergence-free initially, it will remain divergence-free. 
\end{proof}

However, what we will prove in this section is that weak solutions of the ideal MHD equations that arise as vanishing viscosity and resistivity limits of Leray-Hopf weak solutions of the viscous and resistive MHD equations, have a divergence-free magnetic field provided it is initially divergence-free. We recall that the viscous and resistive MHD equations are given by
\begin{align}
&\partial_t u - \nu_1 \Delta u + u \cdot \nabla u - B \cdot \nabla B + \nabla p = 0, \label{viscousMHD1} \\ 
&\partial_t B - \nu_2 \Delta B + u \cdot \nabla B - B \cdot \nabla u = 0, \label{viscousMHD2} \\
&\nabla \cdot u = 0, \label{viscousMHD3}
\end{align}
where $\nu_1 > 0$ is the viscosity and $\nu_2 > 0$ is the resistivity, supplemented by the divergence-free initial data $u \lvert_{t = 0} = u_0$ and $B \lvert_{t = 0} = B_0$. Note that also here we have omitted the divergence-free condition for the magnetic field. The existence of Leray-Hopf weak solutions to the MHD equations under the imposed assumption that the magnetic field is divergence-free is well-known, see for example \cite{sermange,wu}. Like we did in the ideal MHD case, we first introduce a more general weak formulation for the viscous and resistive MHD equations \eqref{viscousMHD1}-\eqref{viscousMHD3} in which potentially $\Div B \neq 0$.
\begin{definition} \label{viscoussolutiondef}
A pair $(u,B)$ is called a Leray-Hopf weak solutions of the viscous and resistive MHD equations if $u, B, \Div B \in L^\infty ((0,T); L^2 (\mathbb{T}^3)) \cap L^2 ((0,T);H^1(\mathbb{T}^3))$ and the following equations are satisfied
\begin{align}
&\int_0^T \int_{\mathbb{T}^3} \bigg[ u \partial_t \psi_1 - \nu_1 \nabla u : \nabla \psi_1 + u \otimes u : \nabla \psi_1 - B \otimes B : \nabla \psi_1 - (\Div B ) B \psi_1 \bigg] dx dt = 0, \label{viscousmhdweakform1} \\
&\int_0^T \int_{\mathbb{T}^3} \bigg[ B \partial_t \psi_2 - \nu_2 \nabla B : \nabla \psi_2 + u \otimes B : \nabla \psi_2 - B \otimes u : \nabla \psi_2 - u (\Div B) \psi_2 \bigg] dx dt = 0, \label{viscousmhdweakform2} \\
&\int_0^T \int_{\mathbb{T}^3} u \cdot \nabla \varphi dx dt = 0, \label{viscousmhdweakform3}
\end{align}
for all divergence-free $\psi_1, \psi_2 \in \mathcal{D} (\mathbb{T}^3 \times (0,T); \mathbb{R}^3)$ and $\varphi \in \mathcal{D} (\mathbb{T}^3 \times (0,T);\mathbb{R})$. Moreover, we assume that the solution attains the initial data $u_0, B_0 \in L^2 (\mathbb{T}^3)$ strongly in $L^2 (\mathbb{T}^3)$, where we also assume that $u_0$ is weakly divergence-free.
\end{definition}
We now prove the global existence of Leray-Hopf weak solutions for general initial vector fields $B_0 \in L^2 (\mathbb{T}^3)$, which are not necessarily divergence-free. Note that we always assume that $u_0$ is weakly divergence-free.
\begin{theorem}
Let $u_0, B_0 \in L^2 (\mathbb{T}^3)$ such that $\Div B_0 \in L^2 (\mathbb{T}^3)$ and $\Div u_0 = 0$. Then there exists a Leray-Hopf weak solution (in the sense of Definition \ref{viscoussolutiondef}) on the time interval $[0,T]$ to the viscous and resistive MHD equations \eqref{viscousMHD1}-\eqref{viscousMHD3}. 
\end{theorem}
\begin{proof}
This result can be established by using the Galerkin method. We will only provide a brief sketch of the proof, as the method is standard. In particular, we will only highlight the aspects which differ from the usual proof for the MHD equations under the imposed divergence-free condition for the magnetic field (see for example \cite[~Theorem 2.2]{wu}). We define $P_n$ to be the $L^2$ orthogonal projection onto the first $n$ Fourier modes. Note that we include the zero mode as part of the projection as the mean is not preserved by the equations. We find that there exists a sequence $\{ (u_n, B_n, p_n) \}_{n=1}^\infty$ satisfying the Galerkin approximation of the equations
\begin{align}
&\partial_t u_n - \nu_1 \Delta u_n + P_n \big[ u_n \cdot \nabla u_n \big] + \nabla p_n - P_n \big[ B_n \cdot \nabla B_n \big] = 0, \label{galerkin1} \\
&\partial_t B_n - \nu_2 \Delta B_n + P_n \big[ u_n \cdot \nabla B_n \big] - P_n \big[ B_n \cdot \nabla u_n \big] = 0, \label{galerkin2} \\
&\nabla \cdot u_n = 0. \label{galerkin3}
\end{align}
We note that in fact $u_n$ lies in a subset of the space spanned by the first $n$ Fourier modes. This subset is spanned by eigenfunctions of the Stokes operator (as $u_n$ is divergence-free). Next we define $u_{0,n} \coloneqq P_n u_0$ and $B_{0,n} \coloneqq P_n B_0$, which are the initial data.  We observe that system \eqref{galerkin1}-\eqref{galerkin3} is a system of ODEs with quadratic nonlinearities, and hence one deduces the local existence and uniqueness of a solution.

By taking the divergence of equation \eqref{galerkin2} and using equation \eqref{galerkin3}, we find that (by recalling that $P_n$ commutes with derivatives)
\begin{equation*}
\partial_t \Div B_n - \nu_2 \Delta \Div B_n + P_n \big[ u_n \cdot \nabla \Div B_n \big] = 0.
\end{equation*}
Then by performing standard energy estimates, we find the following identities 
\begin{align}
&\frac{1}{2} \frac{d}{dt} \lVert u_n (\cdot, t) \rVert_{L^2}^2 + \frac{1}{2} \frac{d}{d t} \lVert B_n (\cdot, t) \rVert_{L^2}^2 + \nu_1 \lVert \nabla u_n (\cdot, t) \rVert_{L^2}^2 + \nu_2 \lVert \nabla B_n (\cdot, t) \rVert_{L^2}^2 = - \int_{\mathbb{T}^3} (u_n \cdot B_n) \Div B_n  dx, \label{Lerayhopfestimate} \\
&\frac{1}{2} \frac{d}{d t} \lVert \Div B_n (\cdot, t) \rVert_{L^2}^2 + \nu_2 \lVert \nabla \Div B_n \rVert_{L^2}^2 = 0. \label{Lerayhopfestimate2}
\end{align}
From \eqref{Lerayhopfestimate2} we obtain (for all $t \in (0,T)$)
\begin{equation} \label{divBenergyestimate}
\lVert \Div B_n ( \cdot, t) \rVert_{L^2}^2 + 2 \nu_2 \int_0^t \lVert \nabla \Div B_n (\cdot, t') \rVert_{L^2}^2 dt' \leq \lVert \Div B_{0,n} \rVert_{L^2}^2 \leq \lVert \Div B_0 \rVert_{L^2}^2.
\end{equation}
We recall the following interpolation inequality
\begin{equation} \label{interpolationinequality}
\lVert f \rVert_{L^4 (L^3)} \leq \lVert f \rVert_{L^\infty (L^2)}^{1/2} \lVert f \rVert_{L^2 (L^6)}^{1/2}.
\end{equation}
Next we derive the following bound (by using the Sobolev embedding theorem)
\begin{align*}
&\int_{\mathbb{T}^3} \left\lvert (u_n \cdot B_n) \Div B_n \right\rvert  dx \leq \lVert u_n \rVert_{L^3} \lVert B_n \rVert_{L^3} \lVert \Div B_n \rVert_{L^6} \\
&\leq \lVert u_n \rVert_{L^2}^{1/2} \lVert u_n \rVert_{L^6}^{1/2} \lVert B_n \rVert_{L^2}^{1/2} \lVert B_n \rVert_{L^6}^{1/2} \lVert \Div B_n \rVert_{L^6} \\
&\leq \frac{\nu_1}{4} \lVert \nabla u_n \rVert_{L^2}^2 + \frac{\nu_2}{4} \lVert \nabla B_n \rVert_{L^2}^2 + \frac{1}{4} \bigg[ \max \bigg\{ \frac{1}{\nu_1}, \frac{1}{\nu_2} \bigg\} \lVert \Div B_n \rVert_{L^6}^2 + \max \{ \nu_1, \nu_2 \} \bigg] \bigg( \lVert u_n \rVert_{L^2}^2 + \lVert B_n \rVert_{L^2}^2 \bigg).
\end{align*}
Combining this bound with \eqref{Lerayhopfestimate} yields (for all $t \in (0,T)$)
\begin{align}
&\frac{d}{dt} \bigg( \lVert u_n (\cdot, t) \rVert_{L^2}^2 + \lVert B_n (\cdot, t) \rVert_{L^2}^2 \bigg) + \frac{3}{2} \nu_1 \lVert \nabla u_n (\cdot, t) \rVert_{L^2}^2 + \frac{3}{2} \nu_2 \lVert \nabla B_n (\cdot, t) \rVert_{L^2}^2 \nonumber \\
&\leq \frac{1}{2} \bigg[ \max \bigg\{ \frac{1}{\nu_1}, \frac{1}{\nu_2} \bigg\} \lVert \Div B_n \rVert_{L^6}^2 + \max \{ \nu_1, \nu_2 \} \bigg] \bigg( \lVert u_n \rVert_{L^2}^2 + \lVert B_n \rVert_{L^2}^2 \bigg). \label{L2estimate}
\end{align}
Thus by using Gronwall's inequality, we find that (by using Poincaré's inequality)
\begin{align*}
&\lVert u_n (\cdot, t) \rVert_{L^2}^2 + \lVert B_n (\cdot, t) \rVert_{L^2}^2 \\ 
&\leq \exp \bigg( \frac{1}{2} \max \bigg\{ \frac{1}{\nu_1}, \frac{1}{\nu_2} \bigg\} \int_0^t \lVert \nabla \Div B_n (\cdot, t') \rVert_{L^2}^2 dt' + \frac{1}{2} \max \{ \nu_1, \nu_2 \} t \bigg) \bigg[ \lVert u_{n,0} \rVert_{L^2}^2 + \lVert B_{n,0} \rVert_{L^2}^2 \bigg] \\
&\leq \exp \bigg( \frac{1}{2 \nu_2} \max \bigg\{ \frac{1}{\nu_1}, \frac{1}{\nu_2} \bigg\} \lVert \Div B_{0} \rVert_{L^2}^2 + \frac{1}{2} \max \{ \nu_1, \nu_2 \} T \bigg) \bigg[ \lVert u_{0} \rVert_{L^2}^2 + \lVert B_{0} \rVert_{L^2}^2 \bigg],
\end{align*}
where we have again used the Sobolev inequality $\lVert f \rVert_{L^6} \leq C \lVert f \rVert_{H^1}$ as well as estimate \eqref{divBenergyestimate}. As a result, we find that combining this estimate with \eqref{L2estimate} yields that
\begin{align*}
&\frac{d}{dt} \bigg( \lVert u_n (\cdot, t) \rVert_{L^2}^2 + \lVert B_n (\cdot, t) \rVert_{L^2}^2 \bigg) + \frac{3}{2} \nu_1 \lVert \nabla u_n (\cdot, t) \rVert_{L^2}^2 + \frac{3}{2} \nu_2 \lVert \nabla B_n (\cdot, t) \rVert_{L^2}^2 \nonumber \\
&\leq \frac{1}{2} \bigg[ \max \bigg\{ \frac{1}{\nu_1}, \frac{1}{\nu_2} \bigg\} \lVert \Div B_n \rVert_{L^6}^2 + \max \{ \nu_1, \nu_2 \} \bigg] \\
&\cdot \exp \bigg( \frac{1}{2 \nu_2} \max \bigg\{ \frac{1}{\nu_1}, \frac{1}{\nu_2} \bigg\} \lVert \Div B_{0} \rVert_{L^2}^2 + \frac{1}{2} \max \{ \nu_1, \nu_2 \} T \bigg) \bigg[ \lVert u_{0} \rVert_{L^2}^2 + \lVert B_{0} \rVert_{L^2}^2 \bigg].
\end{align*}
Therefore after integrating in time, we find that
\begin{align*}
&\lVert u_n (\cdot, t) \rVert_{L^2}^2 + \lVert B_n (\cdot, t) \rVert_{L^2}^2 + \frac{3}{2} \int_0^t \bigg[ \nu_1 \lVert \nabla u_n (\cdot, t') \rVert_{L^2}^2 + \nu_2 \lVert \nabla B_n (\cdot, t') \rVert_{L^2}^2 \bigg] dt' \leq \lVert u_0 \rVert_{L^2}^2 + \lVert B_0 \rVert_{L^2}^2 \\
&+\frac{1}{2} \bigg[ \max \bigg\{ \frac{1}{\nu_1}, \frac{1}{\nu_2} \bigg\} \int_0^t \lVert \nabla \Div B_n (\cdot, t') \rVert_{L^2}^2 dt' + \max \{ \nu_1, \nu_2 \} t \bigg] \\
&\times \exp \bigg( \frac{1}{2 \nu_2} \max \bigg\{ \frac{1}{\nu_1}, \frac{1}{\nu_2} \bigg\} \lVert \Div B_{0} \rVert_{L^2}^2 + \frac{1}{2} \max \{ \nu_1, \nu_2 \} T \bigg) \bigg[ \lVert u_{0} \rVert_{L^2}^2 + \lVert B_{0} \rVert_{L^2}^2 \bigg].
\end{align*}
Then again by using estimate \eqref{divBenergyestimate}, we find (for all $t \in (0,T)$)
\begin{align}
&\big( \lVert u_n (\cdot, t) \rVert_{L^2}^2 + \lVert B_n (\cdot, t) \rVert_{L^2}^2 \big) + \frac{3}{2} \int_0^t \bigg[ \nu_1 \lVert \nabla u_n (\cdot, t') \rVert_{L^2}^2 + \nu_2 \lVert \nabla B_n (\cdot, t') \rVert_{L^2}^2 \bigg] dt' \nonumber \\
&\leq\bigg[ \bigg(\frac{1}{2 \nu_2} \max \bigg\{ \frac{1}{\nu_1}, \frac{1}{\nu_2} \bigg\} \lVert \Div B_{0} \rVert_{L^2}^2 + \frac{1}{2} \max \{ \nu_1, \nu_2 \} t \bigg) \nonumber \\
&\times \exp \bigg( \frac{1}{2 \nu_2} \max \bigg\{ \frac{1}{\nu_1}, \frac{1}{\nu_2} \bigg\} \lVert \Div B_{0} \rVert_{L^2}^2 + \frac{1}{2} \max \{ \nu_1, \nu_2 \} T \bigg) + 1\bigg] \bigg[ \lVert u_{0} \rVert_{L^2}^2 + \lVert B_{0} \rVert_{L^2}^2 \bigg]. \label{H1estimate}
\end{align}
From this estimate we then conclude that the Galerkin system \eqref{galerkin1}-\eqref{galerkin3} in fact has a unique solution on any finite time interval $[0,T]$. 

Now we proceed to construct a solution of the original system. By estimate \eqref{H1estimate} we conclude that there exist subsequences $\{ u_n \}$ and $\{ B_n \}$ which converge weak-$*$ to limits $u$ and $B$ in $L^\infty ((0,T); L^2 (\mathbb{T}^3))$ and weakly to the same limits in $L^2 ((0,T); H^1 (\mathbb{T}^3))$. Moreover, by using estimate \eqref{divBenergyestimate}, we conclude that there exists a subsequence $\{ \Div B_n \}$ converging weakly to $\Div B$ in $L^2 ((0,T); H^1 (\mathbb{T}^3))$ and weak-$*$ in $L^\infty ((0,T); L^2 (\mathbb{T}^3))$. We first observe that the weak(-$*$) limit $\Div B$ coincides with the weak divergence of the weak(-$*$) limit $B$. This can be seen by noticing that for any $n \in \mathbb{N}$ we have that
\begin{equation*}
\int_0^T \int_{\mathbb{T}^3} \Div B_n \Phi dx dt = - \int_0^T \int_{\mathbb{T}^3} B_n \cdot \nabla \Phi dx dt,
\end{equation*}
for all $\Phi \in \mathcal{D} (\mathbb{T}^3 \times (0,T))$. Then by passing to the weak limits in the integrals as $n \rightarrow \infty$, we conclude that $\Div B$ is the weak divergence of $B$.

Next we will use the Aubin-Lions lemma to obtain strong convergence. To this end, we will show that $\partial_t u_n , \partial_t B_n , \partial_t \Div B_n \in L^{4/3} ((0,T); H^{-1} (\mathbb{T}^3))$. We will only estimate the terms which are different compared to the standard proof for the MHD equations, which can be found for example in \cite[~Theorem 2.2]{wu}. We have that
\begin{align*}
\lVert B_n \cdot \nabla B_n \rVert_{H^{-1}} &\leq \lVert B_n \cdot \nabla B_n \rVert_{L^{6/5}} \leq \lVert B_n \rVert_{L^3} \lVert \nabla B_n \rVert_{L^2}.
\end{align*}
Hence we obtain by using Hölder's inequality and inequality \eqref{interpolationinequality}
\begin{align*}
\lVert B_n \cdot \nabla B_n \rVert_{L^{4/3} (H^{-1})} &\leq \bigg( \int_0^T \lVert B_n \rVert_{L^3}^{4/3} \lVert \nabla B_n \rVert_{L^2}^{4/3} dt \bigg)^{3/4} \leq \bigg( \int_0^T \lVert \nabla B_n \rVert_{L^2}^2 dt \bigg)^{1/2} \bigg( \int_0^T \lVert B_n \rVert_{L^3}^4 dt \bigg)^{1/4} \\
&\leq \lVert B_n \rVert_{L^2 (H^1)}^{3/2} \lVert B_n \rVert_{L^\infty (L^2)}^{1/2} \leq \lVert B_0 \rVert_{L^2}^2.
\end{align*}
Similarly, one can show that $B_n \cdot \nabla u_n, u_n \cdot \nabla \Div B_n \in L^{4/3} ((0,T); H^{-1} (\mathbb{T}^3))$. Therefore we conclude that $\partial_t u_n, \partial_t B_n, \partial_t \Div B_n \in L^{4/3} ((0,T); H^{-1} (\mathbb{T}^3))$. Therefore by applying the Aubin-Lions lemma, we obtain (potentially by passing to a different subsequence) that $u_n \rightarrow u$, $B_n \rightarrow B$ and $\Div B_n \rightarrow \Div B$ in $L^2 (\mathbb{T}^3 \times (0,T))$.

We can then conclude that $(u, B, \Div B)$ is a Leray-Hopf weak solution in the sense of Definition \ref{viscoussolutiondef} by passing to the limit for all terms in the weak formulation in equations \eqref{viscousmhdweakform1}-\eqref{viscousmhdweakform3}, where we use the strong convergence in $L^2 (\mathbb{T}^3 \times (0,T))$ to conclude that the nonlinear terms converge. 
\end{proof}
Next we show that the magnetic field of Leray-Hopf weak solutions with divergence-free initial data, will remain divergence-free.
\begin{theorem} \label{lerayhopfdivergence}
Any Leray-Hopf weak solution (in the sense of Definition \ref{viscoussolutiondef}) with the property that $\Div B_0 = 0$ satisfies $\Div B \equiv 0$.
\end{theorem}
\begin{proof}
We first observe that for a Leray-Hopf weak solution (in the sense of Definition \ref{viscoussolutiondef}) $\Div B$ satisfies the following advection-diffusion equation (in a weak sense)
\begin{equation} \label{advectiondiffdivB}
\partial_t \Div B - \nu_2 \Delta \Div B + u \cdot \nabla \Div B = 0.
\end{equation}
This can be seen by proceeding as in Remark \ref{divtransportremark}, i.e. by taking a test function $\nabla \Phi$ (for $\Phi \in \mathcal{D} (\mathbb{T}^3 \times (0,T); \mathbb{R})$) in equation \eqref{viscousmhdweakform2} and integrating by parts. This gives the following equation
\begin{equation*}
\int_0^T \int_{\mathbb{T}^3} \bigg[ \Div B \partial_t \Phi - \nu_2 \nabla \Div B \cdot \nabla \Phi + (u \Div B) \cdot \nabla \Phi \bigg] dx dt = 0,
\end{equation*}
which is indeed a weak formulation of the advection-diffusion equation. Note that the manipulations are justified because $\Div B \in L^\infty ((0,T);L^2 (\mathbb{T}^3)) \cap L^2 ((0,T);H^1(\mathbb{T}^3))$. 

In the terminology of \cite{bonicatto2023}, $\Div B \equiv 0$ is a parabolic solution (i.e. it belongs $L^\infty ((0,T);L^2 (\mathbb{T}^3)) \cap L^2 ((0,T);H^1(\mathbb{T}^3))$) of equation \eqref{advectiondiffdivB}. By Theorem 2.7 in \cite{bonicatto2023}, there exists a unique parabolic solution to equation \eqref{advectiondiffdivB} if $u \in L^2 (\mathbb{T}^3 \times (0,T))$. As this condition is satisfied for Leray-Hopf weak solutions, we conclude that the magnetic field must remain divergence-free if it is divergence-free initially. Notice that it is crucial for this argument that $\Div B \in L^2 ((0,T); H^1 (\mathbb{T}^3))$.
\end{proof}
\begin{remark}
We note that Theorem \ref{lerayhopfdivergence} can be extended to include a forcing in the induction equation of the viscous and resistive MHD equations, but it is crucial that this forcing is divergence-free. Otherwise, even for data with divergence-free initial magnetic field for which a smooth solution exists, the magnetic field need not remain divergence-free over time. Moreover, we also remark that once one establishes that $\Div B = 0$ for any given Leray-Hopf weak solution, then one can conclude that the spatial means of $u$ and $B$ are invariant.
\end{remark}
\begin{remark}
By using the same approach as the proof of Theorem 2.7 in \cite{bonicatto2023}, one finds that for Leray-Hopf weak solutions (in the sense of Definition \ref{viscoussolutiondef}) with $\Div B_0 \neq 0$ that $\Div B$ satisfies the energy equality (for almost all $t \in (0,T)$)
\begin{equation}
\frac{1}{2} \lVert \Div B (\cdot, t) \rVert_{L^2}^2 + \nu_2 \int_0^t \lVert \nabla \Div B (\cdot, t') \rVert_{L^2}^2 dt' = \frac{1}{2} \lVert \Div B_0 \rVert_{L^2}^2.
\end{equation}
\end{remark}
Finally, we show that weak solutions of the ideal MHD equations that arise as weak-$*$ limits of Leray-Hopf weak solutions, have the property that the divergence-free condition for $B$ is preserved over time.
\begin{corollary} \label{viscositylimitcorollary}
Any weak solution $(u,B)$ of the ideal MHD equations (in the sense of Definition \ref{newweaksolMHDdef}) which satisfies $\Div u_0 = \Div B_0 = 0$ and arises as a vanishing viscosity and resistivity weak-$*$ limit in $L^\infty ((0,T); L^2 (\mathbb{T}^3))$ of a sequence of Leray-Hopf weak solutions with divergence-free initial data for the magnetic field, will have the property that $\nabla \cdot B = 0$ in a weak sense. 
\end{corollary}
\begin{proof}
Let $\{ (u^\nu, B^\nu, \Div B^\nu ) \}$ be a sequence of Leray-Hopf weak solutions with corresponding viscosities and resistivities $\nu \rightarrow 0$, which converges weak-$*$ to $(u, B, \Div B)$ in $L^\infty ((0,T); L^2 (\mathbb{T}^3))$. Note that $L^\infty ((0,T); L^2 (\mathbb{T}^3))$ is the natural function space to consider weak-$*$ convergence, as the conservation laws (for the energy) imply bounds on weak solutions at this regularity level. For any $\nu > 0$ and for all $\Phi \in \mathcal{D} (\mathbb{T}^3 \times (0,T))$, by Theorem \ref{lerayhopfdivergence} it holds that
\begin{equation*}
\int_0^T \int_{\mathbb{T}^3} B^\nu \cdot \nabla \Phi dx dt = 0.
\end{equation*}
Then by passing to the weak-$*$ limit in this integral we conclude that
\begin{equation*}
\int_0^T \int_{\mathbb{T}^3} B \cdot \nabla \Phi dx dt = 0,
\end{equation*}
and hence we conclude that $B$ is weakly divergence-free. 
\end{proof}
\section{Conclusion} \label{conclusion}
In this work, we have studied the conservation of helicity. To establish sufficient conditions for helicity conservation we have extended a new definition of weak solutions that was originally introduced in \cite{boutroshydrostatic} to the hydrostatic Euler equations. The solutions introduced in Definition \ref{functionalvorticity} in this paper have been called functional vorticity solutions. For weak solutions of this type the multiplication between the velocity and vorticity fields in the nonlinear terms in the vorticity equation is interpreted as a paraproduct when the vorticity is only a functional (with respect to the spatial variables). 

This approach has made it possible to establish an equation of local helicity balance for the Euler equations. To the best knowledge of the authors, this allows for the first rigorous definition of local helicity fluxes and densities at this low level of regularity. Moreover, a sufficient condition for helicity conservation was obtained which is more general than many of the existing criteria in the literature. 

Based on the local helicity balance, a scaling law for a third-order structure function was proven. In addition, a sufficient condition for helicity conservation in the inviscid limit was derived. It would be interesting to see if there are further connections between the approach taken in this paper of using functional vorticity solutions and existing results in the study of helical turbulence. 

In addition, we have studied the MHD equations and the kinematic dynamo model. Using paradifferential calculus, we have established new sufficient criteria for weak solutions to conserve magnetic helicity. We have also introduced an analogue of functional vorticity solutions for the kinematic dynamo model. 

Finally, we have studied the problem of whether the magnetic field remains divergence-free under the evolution of the MHD equations. We have proven the global existence of Leray-Hopf weak solutions to include the case when the magnetic field is not divergence-free. We then proved that for Leray-Hopf weak solutions for which the initial data of the magnetic field is divergence-free, the magnetic field will remain divergence-free. Then we used this property to conclude that weak solutions of the ideal MHD equations which arise as vanishing viscosity and resistivity limits, will have the feature that the divergence-free property of the magnetic field is preserved in time.
\section*{Acknowledgements}
D.W.B. acknowledges support from the Cambridge Trust, the Cantab Capital Institute for Mathematics of Information and the Prince Bernhard Culture fund. The authors have benefited from the inspiring environment of the CRC
1114 ``Scaling Cascades in Complex Systems'', Project Number 235221301,
Project C06, funded by Deutsche Forschungsgemeinschaft (DFG). Moreover,
this work was also supported in part by the DFG Research Unit FOR 5528 on
Geophysical Flows.
\begin{appendices}
\section{Paradifferential calculus inequalities} \label{paradifferentialappendix}
For the convenience of the reader, we briefly recall some inequalities from Bony's paradifferential calculus. Further details can be found in \cite{bahouri,boutroshydrostatic}.
\begin{lemma} \label{paraproductlemma}
Let $1 \leq p_1, p_2, p, q_1, q_2, q \leq \infty$, $\beta < 0 < \alpha$, $\theta > 0$ and
\begin{equation*}
\frac{1}{p} = \frac{1}{p_1} + \frac{1}{p_2}.
\end{equation*}
Then we have the following inequalities:
\begin{itemize}
    \item If $f \in B^\alpha_{p_1,q_1} (\mathbb{T}^3)$, $g \in B^\beta_{p_2,q_2} (\mathbb{T}^3)$ and $\alpha + \beta > 0$ then
    \begin{equation} \label{paradiffineq1}
        \lVert fg \rVert_{B^\beta_{p,q_2}} \lesssim \lVert f \rVert_{B^\alpha_{p_1,q_1}} \lVert g \rVert_{B^\beta_{p_2,q_2}}.
    \end{equation}
    \item For $f \in B^\alpha_{p_1,q} (\mathbb{T}^3)$ and $g \in B^\alpha_{p_2,q} (\mathbb{T}^3)$ we have that
    \begin{equation} \label{paradiffineq2}
        \lVert fg \rVert_{B^\alpha_{p,q}} \lesssim \lVert f \rVert_{B^\alpha_{p_1,q}} \lVert g \rVert_{B^\alpha_{p_2,q}}.
    \end{equation}
    \item For $f \in B^\theta_{p_1,q_1} (\mathbb{T}^3) \cap B^{\alpha + \theta}_{p_2,q_2} (\mathbb{T}^3)$
    \begin{equation} \label{paradiffineq3}
        \lVert f^2 \rVert_{B^\alpha_{p,q}} \lesssim \lVert f \rVert_{B^\theta_{p_1,q_1}} \lVert f \rVert_{B^{\alpha+\theta}_{p_2,q_2}}.
    \end{equation}
\end{itemize}
\end{lemma}
In addition, we recall the following lemma.
\begin{lemma}[Lemma 4.13 in \cite{boutroshydrostatic}] \label{diffquotientlemma}
Let $s, \epsilon > 0$ such that $s + \epsilon < 1$, and let $1 \leq p \leq \infty$. Then for all $f \in B^{s+\epsilon}_{p,\infty} (\mathbb{T}^3)$ and $\xi \in \mathbb{R}^3 \backslash \{ 0 \}$ it holds that
\begin{equation}
\lVert \delta f (\xi;\cdot) \rVert_{B^s_{p,\infty}} \lesssim \lvert \xi \rvert^\epsilon \lVert f \rVert_{B^{s+\epsilon}_{p,\infty}}.
\end{equation}
\end{lemma}
Now we prove mollification estimates for the case of negative Besov spaces.
\begin{lemma} \label{mollificationbesovlemma}
Let $0 < \alpha_1 < \alpha_2 < 1$ and let $1 \leq p, q < \infty$. Then if $f \in B^{-\alpha_1}_{p,q} (\mathbb{T}^3)$ it satisfies (where we recall that $f^\epsilon$ is the mollification of $f$)
\begin{equation} \label{mollificationbesovineq1}
\lim_{\epsilon \rightarrow 0} \lVert f^\epsilon - f  \rVert_{B^{-\alpha_1}_{p,q}} = 0.
\end{equation}
Moreover, the function $f$ also satisfies the following inequality
\begin{equation} \label{mollificationbesovineq2}
\lVert f^\epsilon - f  \rVert_{B^{-\alpha_2}_{p,1}} \lesssim \epsilon^{\alpha_2 - \alpha_1} \lVert f \rVert_{B^{-\alpha_1}_{p,\infty}}.
\end{equation}
In addition, if $p \in (1,\infty)$ it holds that
\begin{equation} \label{mollificationbesovineq3}
\lVert f^\epsilon \rVert_{L^p} \lesssim \epsilon^{-\alpha_1} \lVert f \rVert_{B^{-\alpha_1}_{p,\infty}}.
\end{equation}
Finally, for any $0 < \alpha_3 < 1 - \alpha_1$ the following inequality also holds
\begin{equation} \label{mollificationbesovineq4}
\lVert f^\epsilon \rVert_{B^{\alpha_3}_{p,1}} \lesssim \epsilon^{-\alpha_1 - \alpha_3} \lVert f \rVert_{B^{-\alpha_1}_{p,\infty}}.
\end{equation}
\end{lemma}
\begin{proof}
We recall that for $q < \infty$ the Besov norm is given by (where $\Delta_j f$ are the Littlewood-Paley blocks of $f$, see \cite{bahouri,boutroshydrostatic} for further details)
\begin{equation*}
\lVert f \rVert_{B^{s}_{p,q}} \coloneqq \lVert \Delta_{-1} f \rVert_{L^p} + \bigg( \sum_{j=0}^\infty 2^{s j q} \lVert \Delta_j f \rVert_{L^p}^q \bigg)^{1/q}.
\end{equation*}
We then find the following
\begin{equation*}
\lVert f^\epsilon - f \rVert_{B^{s}_{p,q}} = \lVert (\Delta_{-1} f)^\epsilon - \Delta_{-1} f \rVert_{L^p} + \bigg( \sum_{j=0}^\infty 2^{s j q} \lVert (\Delta_j f)^\epsilon - \Delta_j f \rVert_{L^p}^q \bigg)^{1/q}.
\end{equation*}
We then recall the following properties for a function $g \in L^p (\mathbb{T}^3)$ with $p < \infty$
\begin{align}
&\lVert g^\epsilon \rVert_{L^p} \lesssim \lVert g \rVert_{L^p}, \label{Lpestimate} \\
&\lim_{\epsilon \rightarrow 0} \, \lVert g^\epsilon - g \rVert_{L^p} = 0. \label{Lpconvergence}
\end{align}
Using inequality \eqref{Lpestimate} we find that
\begin{equation*}
\lVert f^\epsilon - f \rVert_{B^{s}_{p,q}} \lesssim 2 \lVert f \rVert_{B^s_{p,q}}.
\end{equation*}
Therefore we can apply the Lebesgue dominated convergence theorem and find by using \eqref{Lpconvergence}
\begin{align*}
\lim_{\epsilon \rightarrow 0} \lVert f^\epsilon - f \rVert_{B^{s}_{p,q}} &= \lim_{\epsilon \rightarrow 0} \lVert (\Delta_{-1} f)^\epsilon - \Delta_{-1} f \rVert_{L^p} + \lim_{\epsilon \rightarrow 0} \bigg( \sum_{j=0}^\infty 2^{s j q} \lVert (\Delta_j f)^\epsilon - \Delta_j f \rVert_{L^p}^q \bigg)^{1/q} \\
&= \bigg( \sum_{j=0}^\infty \big[ \lim_{\epsilon \rightarrow 0} \big(2^{s j q} \lVert (\Delta_j f)^\epsilon - \Delta_j f \rVert_{L^p}^q \big) \big] \bigg)^{1/q} = 0,
\end{align*}
which proves \eqref{mollificationbesovineq1} for any $\alpha_1 \in \mathbb{R}$.

In order to prove inequality \eqref{mollificationbesovineq2}, we recall that by Proposition 2.76 in \cite{bahouri} we have the following bound on the Besov norm (where $Q^{\alpha}_{p,q}$ is the space of functions $\rho \in \mathcal{D} (\mathbb{T}^3)$ such that $\lVert \rho \rVert_{B^{\alpha}_{p,q}} \leq 1$)
\begin{equation} \label{besovnormbound}
\lVert f \rVert_{B^{-\alpha_2}_{p,q}} \lesssim \sup_{\rho \in Q^{\alpha_2}_{p',q'}} ( f , \rho )_{B^{-\alpha_2}_{p,q} \times B^{\alpha_2}_{p',q'}},
\end{equation}
where $p'$ and $q'$ are the Hölder conjugates of $p$ and $q$. The duality bracket $(\cdot, \cdot)$ is defined as follows 
\begin{equation}
(f,\rho)_{B^{-\alpha_2}_{p,q} \times B^{\alpha_2}_{p',q'}} \coloneqq \sum_{\lvert j - j' \rvert \leq 1} \int_{\mathbb{T}^3} \Delta_j f \Delta_{j'} \rho dx. 
\end{equation}
We also recall that in \cite[~Proposition 2.76]{bahouri} it is shown that $(\cdot, \cdot)$ is a continuous bilinear functional. In order to prove inequality \eqref{mollificationbesovineq2}, we first need to show that for all $g \in B^{\alpha_2}_{p',\infty} (\mathbb{T}^3)$ we have
\begin{equation} \label{mollificationbesov}
\lVert g^\epsilon - g \rVert_{B^{\alpha_1}_{p',1}} \lesssim \epsilon^{\alpha_2 - \alpha_1} \lVert g \rVert_{B^{\alpha_2}_{p',\infty}}.
\end{equation}
This inequality can be established by using Theorem 2.80 in \cite{bahouri} and hence find
\begin{align*}
\lVert g^\epsilon - g \rVert_{B^{\alpha_1}_{p',1}} \lesssim \lVert g^\epsilon - g \rVert_{B^0_{p',\infty}}^{1 - \frac{\alpha_1}{\alpha_2}} \lVert g^\epsilon - g \rVert_{B^{\alpha_2}_{p',\infty}}^{\frac{\alpha_1}{\alpha_2}} \lesssim \epsilon^{\alpha_2 - \alpha_1} \lVert g \rVert_{B^{\alpha_2}_{p',\infty}},
\end{align*}
where we have used Lemma \ref{diffquotientlemma}, or alternatively inequality (7) from \cite{constantin}, as well as Proposition 17.12 from \cite{leoni}. Now by using inequalities \eqref{besovnormbound} and \eqref{mollificationbesov} we find
\begin{align*}
\lVert f^\epsilon - f  \rVert_{B^{-\alpha_2}_{p,1}} &\lesssim \sup_{\rho \in Q^{\alpha_2}_{p',\infty}} ( f^\epsilon - f , \rho )_{B^{-\alpha_2}_{p,1} \times B^{\alpha_2}_{p',\infty}} = \sup_{\rho \in Q^{\alpha_2}_{p',\infty}} ( f, \rho^\epsilon - \rho )_{B^{-\alpha_1}_{p,\infty} \times B^{\alpha_1}_{p',1}} \\
&\lesssim \lVert f \rVert_{B^{-\alpha_1}_{p,\infty}} \sup_{\rho \in Q^{\alpha_2}_{p',\infty}} \lVert \rho^\epsilon - \rho \rVert_{B^{\alpha_1}_{p',1}} \lesssim \epsilon^{\alpha_2 - \alpha_1} \lVert f \rVert_{B^{-\alpha_1}_{p,\infty}} \sup_{\rho \in Q^{\alpha_2}_{p',\infty}} \lVert \rho \rVert_{B^{\alpha_2}_{p',\infty}} \lesssim \epsilon^{\alpha_2 - \alpha_1} \lVert f \rVert_{B^{-\alpha_1}_{p,\infty}},
\end{align*}
which concludes the proof of inequality \eqref{mollificationbesovineq2}. Now we establish inequality \eqref{mollificationbesovineq3}. We define $Q^p$ to be the space of functions $\rho \in L^p (\mathbb{T}^3)$ such that $\lVert \rho \rVert_{L^p} \leq 1$. By the duality characterisation of $L^p (\mathbb{T}^3)$ for $p \in (1,\infty)$ we find (where we use Lemma A.3 in \cite{boutrosnonuniqueness}, Theorem C.16 in \cite{leoni} and also Theorem 2.80 in \cite{bahouri})
\begin{align*}
\lVert f^\epsilon \rVert_{L^p} &= \sup_{\rho \in Q^{p'}} (f^\epsilon, \rho)_{L^p \times L^{p'}} = \sup_{\rho \in Q^{p'}} (f, \rho^\epsilon)_{B^{-\alpha_1}_{p,\infty} \times B^{\alpha_1}_{p',1}} \lesssim \sup_{\rho \in Q^{p'}} \big[ \lVert f \rVert_{B^{-\alpha_1}_{p,\infty}} \lVert \rho^\epsilon \rVert_{B^{\alpha_1}_{p',1}} \big] \\
&\lesssim \sup_{\rho \in Q^{p'}} \big[ \lVert f \rVert_{B^{-\alpha_1}_{p,\infty}} \lVert \rho^\epsilon \rVert_{B^0_{p',\infty}}^{1-\alpha_1} \lVert \rho^\epsilon \rVert_{B^1_{p',\infty}}^{\alpha_1} \big] \lesssim \sup_{\rho \in Q^{p'}} \big[ \lVert f \rVert_{B^{-\alpha_1}_{p,\infty}} \lVert \rho^\epsilon \rVert_{L^{p'}}^{1-\alpha_1} \lVert \rho^\epsilon \rVert_{W^{1,p'}}^{\alpha_1} \big] \\
&\lesssim \epsilon^{-\alpha_1} \lVert f \rVert_{B^{-\alpha_1}_{p,\infty}} \sup_{\rho \in Q^{p'}} \lVert \rho^\epsilon \rVert_{L^{p'}} \lesssim \epsilon^{-\alpha_1} \lVert f \rVert_{B^{-\alpha_1}_{p,\infty}},
\end{align*}
which is what we had to show. Finally, we turn to the proof of inequality \eqref{mollificationbesovineq4}. We first choose $\theta \in (0, \alpha_1)$ such that $\theta + \alpha_1 + \alpha_3 < 1$. Now by using equation (A.8) in \cite{boutrosfractional} we have that
\begin{align*}
\lVert f^\epsilon \rVert_{B^{\alpha_3}_{p,1}} &\simeq \lVert (I + (- \Delta)^{1/2})^{\alpha_3 + \theta} f^\epsilon \rVert_{B^{-\theta}_{p,1}} \\
&\lesssim \lVert (I + (- \Delta)^{1/2})^{\alpha_3 + \theta} f^\epsilon \rVert_{B^0_{p,\infty}}^{1 - \frac{\theta}{\alpha_1+\alpha_3+\theta}} \lVert (I + (- \Delta)^{1/2})^{\alpha_3 + \theta} f^\epsilon \rVert_{B^{-\alpha_1-\alpha_3-\theta}_{p,\infty}}^{\frac{\theta}{\alpha_1+\alpha_3+\theta}} \\
&\lesssim \epsilon^{-(\alpha_1+\alpha_3+\theta) \cdot \left( 1 - \frac{\theta}{\alpha_1+\alpha_3+\theta} \right)} \lVert (I + (- \Delta)^{1/2})^{\alpha_3 + \theta} f \rVert_{B^{-\alpha_1-\alpha_3-\theta}_{p,\infty}} \simeq \epsilon^{-\alpha_1 - \alpha_3} \lVert f \rVert_{B^{-\alpha_1}_{p,\infty}},
\end{align*}
where we have used Theorem 2.80 in \cite{bahouri}, the embedding $L^p (\mathbb{T}^3) \subset B^0_{p,\infty} (\mathbb{T}^3)$ as well as inequality \eqref{mollificationbesovineq3}.
\end{proof}
\end{appendices}
\footnotesize
\bibliographystyle{acm}
\bibliography{On_the_conservation_of_helicity}

\end{document}